\documentclass{article}
\pdfoutput=1
\usepackage{euscript,amsmath,amssymb,amsfonts,graphicx,bm,amsthm}
\usepackage{array}

  \usepackage{paralist}
  \usepackage{graphics} 
\usepackage[caption=false]{subfig}
\usepackage{soul}

\usepackage{latexsym,amssymb,enumerate,amsmath,amsthm} 
  \usepackage{graphicx} 
  \usepackage{tikz}
     \usepackage[colorlinks=false]{hyperref}
 \hypersetup{urlcolor=blue, citecolor=red}
\usepackage{latexsym,amssymb,enumerate,amsmath} 
\usepackage{pgfplots}
\usepackage{soul,xcolor}

\usepackage{mathtools}

\newcommand{\dd}{\text{d}}
\def\eps{\varepsilon}
\def\E{\mathbb{E}}
\def\P{\mathbb{P}}
\def\R{\mathbb{R}}

\def\dist{\textup{d}}
\def\O{\mathcal{O}}

\def\td{t_{\textup{diff}}}
\def\x{\textbf{x}}
\def\n{\textbf{n}}
\def\L{\mathcal{L}}
\def\kon{k_{\text{on}}}
\def\S{\mathbf{S}}
\def\SS{\widetilde{S}}
\def\imp{\textup{imp}}
\def\unif{\textup{unif}}
\def\neu{\textup{neu}}

\newtheorem{theorem}{Theorem}
\newtheorem{proposition}[theorem]{Proposition}

\theoremstyle{plain}
\theoremstyle{remark}
\newtheorem{remark}[theorem]{Remark}
\theoremstyle{definition}
\newtheorem{definition*}{Definition}

\begin{document}


\title{Competition between slow and fast regimes for extreme first passage times of diffusion}


\author{Jacob B. Madrid \and Sean D. Lawley\thanks{Department of Mathematics, University of Utah, Salt Lake City, UT 84112 USA (\texttt{lawley@math.utah.edu}). The authors were supported by the National Science Foundation (Grant Nos.\ DMS-1944574, DMS-1814832, and DMS-1148230).}
}
\date{\today}
\maketitle

\begin{abstract}
Many physical, chemical, and biological systems depend on the first passage time (FPT) of a diffusive searcher to a target. Typically, this FPT is much slower than the characteristic diffusion timescale. For example, this is the case if the target is small (the narrow escape problem) or if the searcher must escape a potential well. However, many systems depend on the first time a searcher finds the target out of a large group of searchers, which is the so-called extreme FPT. Since this extreme FPT vanishes in the limit of many searchers, the prohibitively slow FPTs of diffusive search can be negated by deploying enough searchers. However, the notion of ``enough searchers'' is poorly understood. How can one determine if a system is in the slow regime (dominated by small targets or a deep potential, for example) or the fast regime (dominated by many searchers)? How can one estimate the extreme FPT in these different regimes? In this paper, we answer these questions by deriving conditions which ensure that a system is in either regime and finding approximations of the full distribution and all the moments of the extreme FPT in these regimes. Our analysis reveals the critical effect that initial searcher distribution and target reactivity can have on extreme FPTs.
\end{abstract}

\section{Introduction}

The first time a diffusive searcher finds a target determines the timescale of many physical, chemical, and biological processes \cite{redner2001}. For example, the ``searcher'' could be an ion, a protein, a sperm cell, or an animal, and the ``target'' could be a membrane channel, a receptor, an egg, or a prey \cite{chou_first_2014}. This random time is called a first passage time (FPT).

In many applications, this FPT is much slower than the characteristic diffusion time. More precisely, let $\tau>0$ denote the random FPT of a single searcher and let $\td:=L^{2}/D$ be the diffusion time, where $L>0$ is some characteristic lengthscale describing the distance the searcher must travel to reach the target and $D>0$ is the searcher diffusivity. It is often the case that $\tau$ is much slower than $\td$,
\begin{align}\label{slow000}
\tau
\gg \td.
\end{align}
Indeed, the following three widely used frameworks are characterized by \eqref{slow000}.

The first framework is the so-called narrow escape problem \cite{holcman2014}, which seeks to determine how long it takes a diffusive searcher to find a small absorbing target(s) in an otherwise reflecting bounded domain (see Figure~\ref{figschem}a). Work on this problem dates back to Helmholtz \cite{helmholtz1860} and Rayleigh \cite{rayleigh1945} in studies of acoustics, but more recent interest has been driven by applications to biology \cite{holcman2014time}, especially molecular and cellular biology. Indeed, the timescales of many cellular processes depend on the arrival of diffusing ligands to small proteins \cite{benichou2008,bressloff13,grebenkov2017}.

\begin{figure}[t]
\centering
\includegraphics[width=1\linewidth]{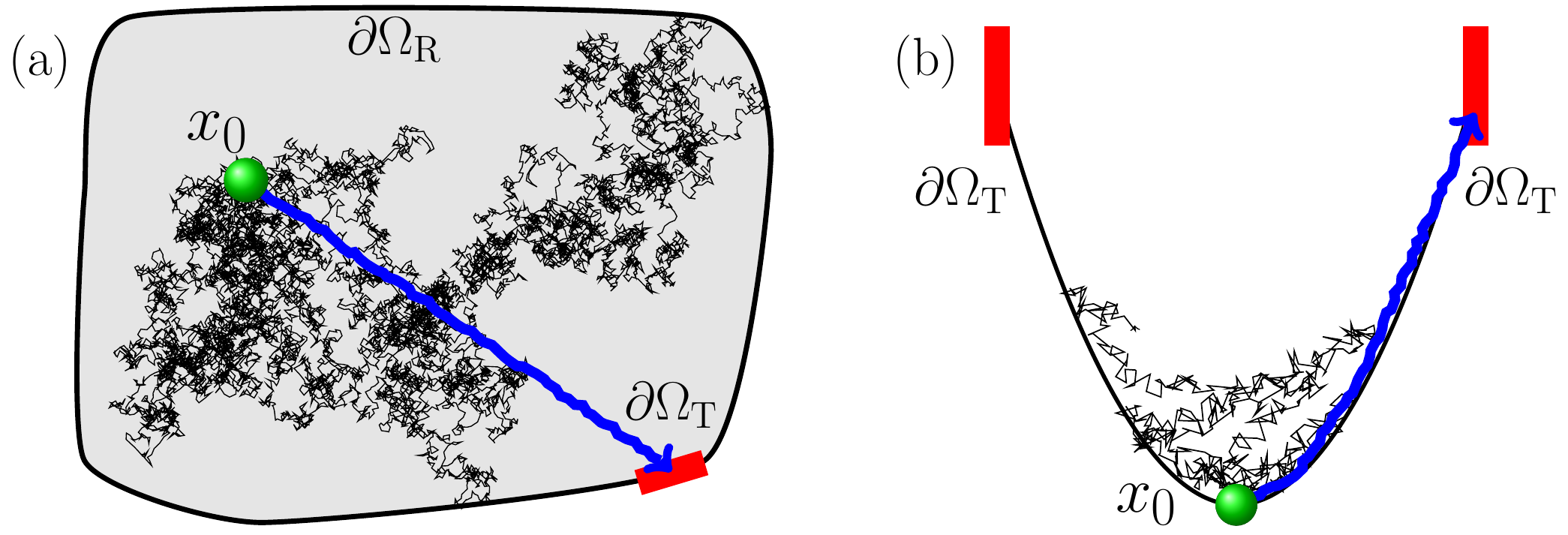}
\caption{
Panel (a) illustrates the narrow escape limit, in which searchers wander around the domain for a long time (thin black trajectory) before finding the target (red window labeled $\partial\Omega_{\text{T}}$). In the many searcher limit, the fastest searchers take an almost direct path (thick blue trajectory) from the starting position (green ball labeled $x_{0}$) to the target. Panel (b) illustrates the analogous situation for searchers that must escape a deep potential well to find the target.
}
\label{figschem}
\end{figure}

A second prototypical scenario which can yield the slow FPT behavior in \eqref{slow000} involves so-called partially absorbing targets \cite{grebenkov2006}. Partially absorbing targets often arise from homogenizing a patchy surface which contains perfectly absorbing targets on an otherwise reflecting surface \cite{Berezhkovskii2004,Muratov2008,cheviakov2012,Dagdug2016,bernoff2018b,lawley2019fpk}. Examples include chemicals binding to cell membrane receptors \cite{berg1977}, reactions on porous catalyst support structures \cite{keil1999}, diffusion current to collections of microelectrodes \cite{scharifker1988}, and water transpiration through plant stomata \cite{brown1900, wolf2016}. Mathematically, partially absorbing targets require a Robin (also called reactive, radiation, or third-type) boundary condition in the corresponding Fokker-Planck equation which involves a reactivity (or trapping rate) parameter.

A third framework characterized by the slow FPT in \eqref{slow000} is when the searcher must escape a potential well to find the target  (see Figure~\ref{figschem}b). This very classical problem arises in Kramers' reaction rate theory and is important for understanding the non-equilibrium behavior of many different physical, chemical, and biological processes \cite{hanggi1990, pollak2005}. 

In each of these three frameworks, there is a natural dimensionless parameter $\eps>0$ characterizing the FPT, $\tau=\tau(\eps)$. In the narrow escape problem, $\eps$ measures the target size. For partially absorbing targets, $\eps$ measures the target reactivity. In the case of escape from a potential well, $\eps$ measures the potential depth. Much of the theoretical work on these three problems has focused on determining how the FPT diverges in the limit $\eps\to0$. Indeed, detailed asymptotic approximations have been developed to understand the following divergence \cite{holcman2013},
\begin{align}\label{slow2}
\frac{\tau(\eps)}{\td}
\to\infty\quad\text{as }\eps\to0.
\end{align}

However, several recent studies and commentaries have announced a significant paradigm shift in understanding the timescales in many biological systems \cite{basnayake2019, schuss2019, coombs2019, redner2019, sokolov2019, rusakov2019, martyushev2019, tamm2019, basnayake2019c}. These works have noted that in many systems, the relevant timescale is not the time it takes a given single searcher to find the target, but rather the time it takes the fastest searcher to find the target out of many searchers. One particularly striking example occurs in human reproduction, in which fertilization is triggered by the first sperm cell to find an egg out of $3\times10^{8}$ sperm cells \cite{meerson2015}. More generally, it is believed that deploying many searchers is a common strategy employed by biological systems in order to overcome the prohibitively slow FPTs associated with diffusive search. Indeed, the recently formulated ``redundancy principle'' posits that the many seemingly redundant copies of an object (cells, proteins, molecules, etc.)\ are not a waste, but rather have the specific function of accelerating activation rates \cite{schuss2019}.

To describe the problem more precisely, let $\tau_{1},\dots,\tau_{N}$ be $N$ independent realizations of some FPT $\tau=\tau(\eps)$. If these represent the respective search times of $N$ searchers, then the fastest searcher finds the target at time
\begin{align*}
T_{N}
=T_{N}(\eps)
:=\min\{\tau_{1},\dots,\tau_{N}\}.
\end{align*}
The time $T_{N}$ is called an extreme statistic or extreme FPT \cite{colesbook}. Importantly, if $N$ is large, then $T_{N}$ is much faster than $\tau$. Indeed, with probability one we have that
\begin{align}\label{vanish}
\frac{T_{N}(\eps)}{\td}\to0\quad\text{as }N\to\infty\text{ for any fixed $\eps>0$}.
\end{align}
Moreover, it was recently shown \cite{lawley2020uni} that if the searchers cannot start arbitrarily close to the target, then the leading order divergence of $T_{N}$ as $N\to\infty$ is completely independent of $\eps$ (for each of the three frameworks above). That is, if $N$ is sufficiently large, then the size of the targets, their reactivity, and the potential have no effect on $T_{N}$. 

Therefore, taking $N\to\infty$ and taking $\eps\to0$ are competing limits. That is, if we fix any $\eps>0$ (meaning any fixed target size, reactivity, or potential) and take $N\to\infty$, then $T_{N}$ vanishes as in \eqref{vanish}. On the other hand, if we fix the number of searchers $N\ge1$ and take $\eps\to0$, then with probability one
\begin{align}\label{grow}
\frac{T_{N}(\eps)}{\td}\to\infty\quad\text{as }\eps\to0\text{ for any fixed $N\ge1$}.
\end{align}
Figure~\ref{figschem} illustrates these two competing limits for the narrow escape problem (panel (a)) and for escape from a potential well (panel (b)). 

Since many systems are described by both $\eps\ll1$ and $N\gg1$ \cite{schuss2019}, this raises several natural questions. How can we determine if a system is in the fast escape regime in \eqref{vanish} or the slow escape regime in \eqref{grow}? How can we approximate the distribution of $T_{N}(\eps)$ in these two regimes? How do these distributions depend on the initial searcher locations, spatial dimension, target size, target reactivity, potential depth, etc.?

In this paper, we answer these questions for a variety of systems. In particular, we derive general criteria to determine if $T_{N}(\eps)$ is either in the fast regime in \eqref{vanish} or the slow regime in \eqref{grow} and approximate the full probability distribution of $T_{N}(\eps)$ in these regimes. Furthermore, this analysis reveals that $T_{N}(\eps)$ does not depend on the initial searcher distribution in the slow regime in \eqref{grow}, but $T_{N}(\eps)$ depends critically on the initial searcher distribution in the fast regime in \eqref{vanish}. Indeed, we find several qualitatively different behaviors of $T_{N}(\eps)$, including $\E[T_{N}(\eps)]/\td$ scaling as 
\begin{align*}
\frac{1}{\ln N},\quad
\frac{\E[\tau]}{\td}\frac{1}{N},\quad
\Big(\frac{\E[\tau]}{\td}\Big)^{4}\frac{1}{N^{2}},\quad
\exp\Big(4\frac{\E[\tau]}{\td}\Big)\frac{1}{N^{2}},\quad\text{as }N\to\infty,
\end{align*}
depending on the initial searcher distribution and other details in the problem. 

The rest of the paper is organized as follows. In Section~\ref{slow}, we analyze the $\eps\to0$ regime of \eqref{grow} for a general class of drift-diffusion processes and apply the results to the narrow escape, partial absorption, and deep potential well problems discussed above. In Section~\ref{fast}, we prove general theorems which give the full distribution and all the moments of $T_{N}(\eps)$ in the $N\to\infty$ regime of \eqref{vanish} based on the short-time distribution of $\tau(\eps)$. In Section~\ref{competition}, we apply the results from Sections~\ref{slow} and \ref{fast} to study the competition between the $\eps\to0$ and $N\to\infty$ limits in some analytically tractable examples. The results of this analysis are confirmed by numerical simulations. We conclude by summarizing our results in Table~\ref{tablesummary}, discussing related work, and highlighting some biological implications. An Appendix collects some proofs and technical points.

\section{Slow escape regime}\label{slow}

Let $\tau>0$ be the FPT for a single diffusive searcher to find a target in a bounded domain. Define the characteristic diffusion timescale,
\begin{align*}
\td
:=\frac{L^{2}}{D},
\end{align*}
where $L>0$ is a characteristic lengthscale describing the size of the domain and $D>0$ is the searcher diffusivity. As we see below, if the mean FPT (MFPT) is much slower than the diffusion time,
\begin{align*}
\E[\tau]
\gg \td,
\end{align*}
then it is generally the case that $\tau$ is approximately exponentially distributed \cite{friedman1976, devinatz1978, schuss1980, williams1982, marchetti1983, day1983}.

Now, it is straightforward to check that the minimum of $N$ independent exponential random variables is also exponential. Therefore, if a single FPT $\tau$ is approximately exponential, then the minimum of $N$ independent realizations of $\tau$,
\begin{align*}
T_{N}
:=\min\{\tau_{1},\dots,\tau_{N}\},
\end{align*}
is also approximately exponentially distributed, at least if $N$ is ``sufficiently small.'' In fact, if $N$ is sufficiently small, then we can approximate the full distribution of the ordered sequence of FPTs,
\begin{align}\label{Tknorder}
T_{1,N}<T_{2,N}<\cdots<T_{N-1,N}<T_{N,N},
\end{align}
where $T_{k,N}$ denotes the $k$th fastest FPT,
\begin{align*}
T_{k,N}
:=\min\big\{\{\tau_{1},\dots,\tau_{N}\}\backslash\cup_{j=1}^{k-1}\{T_{j,N}\}\big\},\quad k\in\{1,\dots,N\},
\end{align*}
where $T_{1,N}:=T_{N}$. In this section, we make these ideas precise, characterize the $N$ ``sufficiently small'' regime, and apply the analysis to some prototypical scenarios. 

\subsection{General mathematical analysis}

We first determine the distribution of the ordered FPTs in \eqref{Tknorder} if the individual FPTs $\{\tau_{n}\}_{n=1}^{N}$ are approximately exponential. We begin by recalling the definition of convergence in distribution.

\begin{definition*}
A sequence of random variables $\{X_{j}\}_{j\ge1}$ \emph{converges in distribution} to a random variable $X$ as $j\to\infty$ if
\begin{align}\label{cddef}
\P(X_{j}\le x)
\to\P(X\le x)\quad\text{as }j\to\infty,
\end{align}
for all points $x\in\R$ such that the function $F(x):=\P(X\le x)$ is continuous. If \eqref{cddef} holds, then we write
\begin{align*}
X_{j}\to_{\dist}X\quad\text{as }j\to\infty.
\end{align*}
\end{definition*}

The following proposition gives the distribution of the ordered FPTs \eqref{Tknorder} if the individual FPTs are approximately exponential. Throughout this work, we write
\begin{align*}
X=_{\textup{d}}\textup{Exp}(t)
\end{align*}
to denote that a random variable $X\ge0$ has an exponential distribution with mean $t>0$ (or equivalently with rate $1/t$), which means $\P(X> x)=e^{-x/t}$ for $x\ge0$.

\begin{proposition}\label{propgen}
Let $\tau=\tau(\eps)$ be a random variable that depends on some parameter $\eps>0$. Assume that there exists a scaling $\lambda=\lambda(\eps)>0$ so that
\begin{align}\label{ac}
\lambda \tau
\to_{\dist}\textup{Exp}(1)\quad\text{as }\eps\to0.
\end{align}
Let $\{\tau_{n}\}_{n=1}^{N}$ be $N\ge1$ independent realizations of $\tau$ and define the $k$th order statistic,
\begin{align*}
T_{k,N}
:=\min\big\{\{\tau_{1},\dots,\tau_{N}\}\backslash\cup_{j=1}^{k-1}\{T_{j,N}\}\big\},\quad k\in\{1,\dots,N\}.
\end{align*}
If $k\in\{1,\dots,N\}$, then
\begin{align}\label{genc0}
\lambda T_{k,N}
&\to_{\dist}\sum_{j=1}^{k}\frac{X_{j}}{N-j+1}\quad\text{as }\eps\to0,
\end{align}
where $\{X_{n}\}_{n=1}^{N}$ are iid with $X_{n}=_{\dist}\textup{Exp}(1)$. In fact, the following $N$-dimensional random variable converges in distribution,
\begin{align}\label{genc}
\begin{split}
&\lambda(T_{1,N},T_{2,N},\dots,T_{N-1,N},T_{N,N})\\
&\quad\to_{\dist}\bigg(\frac{X_{1}}{N},
\frac{X_{1}}{N}+\frac{X_{2}}{N-1},\dots,
\sum_{j=1}^{N}\frac{X_{j}}{N-j+1}\bigg)\quad\text{as }\eps\to0.
\end{split}
\end{align}
\end{proposition}

In words, the distribution corresponding to \eqref{genc0}-\eqref{genc} in Proposition~\ref{propgen} means that the times $\{T_{k,N}\}_{k=1}^{N}$ arrive according to a Poisson process with rate $\lambda (N-k)$ between the $k$th and $(k+1)$st arrivals (as $N\to\infty$). The proof of Proposition~\ref{propgen} is given in the Appendix.

\subsection{General spectral expansion}\label{spectral}

Proposition~\ref{propgen} implies that if the parameter regime is such that a single FPT $\tau$ is approximately exponential with rate $\lambda>0$, then the ordered sequence of FPTs in \eqref{Tknorder} has the distribution in \eqref{genc0}-\eqref{genc}, at least if $N$ is sufficiently small. In particular, the fastest FPT, $T_{N}=T_{1,N}$, is well-approximated by an exponential random variable with rate $N\lambda$. We now estimate when this approximation breaks down as $N$ increases.

In order to answer this question, we need information about the rate of convergence in \eqref{ac} for a single FPT. Now, it is often the case that the survival probability of a single FPT of a diffusion process can be expressed in terms of an eigenfunction expansion of the associated backward Kolmogorov equation. In this section, we describe this general situation to obtain a form for the convergence rate in \eqref{ac}.

Let $\Omega\subset\R^{d}$ be a bounded $d$-dimensional spatial domain with $d\ge1$. Assume that the boundary of the domain, $\partial\Omega$, contains a distinguished region(s), $\partial\Omega_{\text{T}}\subseteq\partial\Omega$, which we call the target, and let $\partial\Omega_{\text{R}}=\partial\Omega\backslash\partial\Omega_{\text{T}}$ denote the rest of the boundary. See Figure~\ref{figschem} for an illustration.

Consider a stochastic process $\{X(t)\}_{t\ge0}$ that diffuses in $\overline{\Omega}$ according to the stochastic differential equation (SDE),
\begin{align}\label{sde}
\dd X(t)
=-\nabla V(X(t))\,\dd t+\sqrt{2D}\,\dd W(t),
\end{align}
with reflecting boundary conditions on $\partial\Omega=\partial\Omega_{\text{T}}\cup\partial\Omega_{\text{R}}$. In \eqref{sde}, the drift term is the gradient of a potential $V:\overline{\Omega}\to\R$ and the noise term involves the diffusivity $D>0$ and a standard $d$-dimensional Brownian motion $\{W(t)\}_{t\ge0}$. Let $\tau>0$ be the first time the diffusion process reaches the target,
\begin{align}\label{tau}
\tau
:=\inf\{t>0:X(t)\in\partial\Omega_{\text{T}}\}.
\end{align}
The survival probability conditioned on the searcher starting position,
\begin{align*}
\S(x,t)
:=\P(\tau>t\,|\,X(0)=x),
\end{align*}
satisfies the backward Kolmogorov (or backward Fokker-Planck) equation \cite{pavliotis2014},
\begin{align}\label{backward}
\begin{split}
\tfrac{\partial}{\partial t}\S
&=\L \S,\quad x\in\Omega,\\
\S
&=0,\quad x\in\partial\Omega_{\text{T}},\\
\tfrac{\partial}{\partial\n}\S
&=0,\quad x\in\partial\Omega_{\text{R}},\\
S
&=1,\quad t=0.
\end{split}
\end{align}
In \eqref{backward}, the differential operator $\L$ is the infinitesimal generator of the SDE in \eqref{sde},
\begin{align*}
\L
=-\nabla V(x)\cdot\nabla+D\Delta,
\end{align*}
and $\frac{\partial}{\partial\n}$ is the derivative with respect to the inward unit normal $\n:\partial\Omega\to\R^{d}$.

For the Boltzmann-type weight function,
\begin{align}\label{rho}
\rho(x)
:=\frac{e^{-V(x)/D}}{\int_{\Omega}e^{-V(y)/D}\,\dd y},
\end{align}
it is straightforward to check that the operator $\L$ is formally self-adjoint on the weighted space of square integrable functions \cite{pavliotis2014},
\begin{align*}
L_{\rho}^{2}(\Omega)
:=\Big\{f:\int_{\Omega}|f(x)|^{2}\rho(x)\,\dd x<\infty\Big\},
\end{align*}
with the boundary conditions in \eqref{backward} and the weighted inner product,
\begin{align*}
(f,g)_{\rho}
:=\int_{\Omega}f(x)g(x)\rho(x)\,\dd x.
\end{align*} 
We thus formally expand the solution to \eqref{backward},
\begin{align}\label{Seig}
\S(x,t)=\sum_{n\ge0}(u_{n},1)_{\rho}e^{-\lambda_{n}t}u_{n}(x),
\end{align}
where
\begin{align}\label{order}
0<\lambda_{0}<\lambda_{1}\le\dots,
\end{align}
are the positive eigenvalues of $-\L$ with eigenfunctions $\{u_{n}(x)\}_{n\ge1}$ satisfying
\begin{align}\label{evalproblem}
\begin{split}
-\L u_{n}
&=\lambda_{n}u_{n},\quad x\in\Omega,\\
u_{n}
&=0,\quad x\in\partial\Omega_{\text{T}},\\
\tfrac{\partial}{\partial\n}u_{n}
&=0,\quad x\in\partial\Omega_{\text{R}},
\end{split}
\end{align}
and which are orthonormal,
\begin{align}\label{orthonormal}
(u_{n},u_{m})_{\rho}
=\delta_{nm}\in\{0,1\},
\end{align}
where $\delta_{nm}$ is the Kronecker delta function ($\delta_{nm}=0$ if $n\neq m$ and $\delta_{nn}=1$). 

If a searcher has initial distribution given by a probability measure $\mu_{0}$,
\begin{align}\label{initial}
\P(X(0)\in B)=\mu_{0}(B)=\int_{B}1\,\dd \mu_{0}(x),\quad B\subset\Omega,
\end{align}
then the survival probability of the FPT $\tau$ in \eqref{tau} is
\begin{align*}
S(t)
&:=\P(\tau>t\,|\,X(0)=_{\dist}\mu_{0})
=\int_{\Omega}\S(x,t)\,\dd\mu_{0}(x).
\end{align*}
The eigenfunction expansion above thus gives a formal representation for $S(t)$ as a sum of decaying exponentials,
\begin{align}\label{form}
S(t)
=\sum_{n\ge0}A_{n}e^{-\lambda_{n}t},
\end{align}
where the coefficients are
\begin{align}\label{As}
A_{n}
:=(u_{n},1)_{\rho}\int_{\Omega}u_{n}(x)\,\dd\mu_{0}(x),
\quad n\ge0.
\end{align}


\subsection{Necessary and sufficient conditions for the slow exponential regime}

In many situations, the FPT in \eqref{tau} with survival probability in \eqref{form} is well-approximated by an exponential random variable with rate given by the principal eigenvalue $\lambda_{0}>0$ \cite{friedman1976, devinatz1978, schuss1980, williams1982, marchetti1983, day1983}. That is,
\begin{align*}
\tau
\approx_{\dist}{\textup{Exp}(1/\lambda_{0})}.
\end{align*}
Indeed, this is generically the case when $\lambda_{0}\td\ll1$ (see below). 

In this section, we therefore assume that (i) $\tau>0$ is a nonnegative random variable with survival probability given by a sum of decaying exponentials as in \eqref{form} and (ii) that
\begin{align}\label{acd}
\lambda_{0}\tau
\to_{\dist}\textup{Exp}(1)\quad\text{as }\eps\to0,
\end{align}
where $\eps>0$ is some dimensionless parameter. Using \eqref{form} and the definition of convergence in distribution in \eqref{cddef}, we have that if $x\ge0$, then \eqref{acd} means that
\begin{align*}
\P(\lambda_{0}\tau>x)
=e^{-x}+e^{-x}\Big[A_{0}-1+\sum_{n\ge1}A_{n}e^{(1-\lambda_{n}/\lambda_{0})x}\Big]
\to e^{-x}\quad\text{as }\eps\to0.
\end{align*}
In particular, if we define the error term
\begin{align*}
\eta(\eps,x)
:=A_{0}-1+\sum_{n\ge1}A_{n}e^{(1-\lambda_{n}/\lambda_{0})x},
\end{align*}
then we are assured that
\begin{align*}
\eta(\eps,x)\to0\quad\text{as }\eps\to0\quad\text{for all }x\ge0.
\end{align*}

Combining the assumption in \eqref{acd} with Proposition~\ref{propgen}, we can immediately conclude that the limiting distributions of the ordered FPTs $T_{k,N}$ for $k\in\{1,\dots,N\}$ as $\eps\to0$ are \eqref{genc0}-\eqref{genc}. However, if we fix $\eps>0$ and take $N$ large, then $T_{k,N}$ might leave the regime in \eqref{genc0}-\eqref{genc}. How large can we take $N$ and still be assured that $T_{k,N}$ is in the regime in \eqref{genc0}-\eqref{genc}? We first consider the case $k=1$. That is, we ask how large can we take $N$ and still guarantee that the fastest FPT $T_{N}=T_{1,N}$ is approximately exponential with rate $N\lambda_{0}>0$.

By definition of $T_{N}$, we have that
\begin{align*}
\P(N\lambda_{0}T_{N}>x)
&=\big[\P(\tau>x/(N\lambda_{0})\big]^{N}\\
&=e^{-x}\Big\{1+\sum_{k=1}^{N}{N\choose k}\big[A_{0}-1+\sum_{n\ge1}A_{n}e^{(1-\lambda_{n}/\lambda_{0})x/N}\big]^{k}\Big\}\\
&=e^{-x}\Big\{1+\sum_{k=1}^{N}{N\choose k}\big[\eta(\eps,x/N)\big]^{k}\Big\}\\
&=e^{-x}\Big\{1+N\eta(\eps,x/N)+\O\big((\eta(\eps,x/N))^{2}\big)\Big\}\quad\text{as }\eps\to0.
\end{align*}
Hence, the regime in \eqref{genc0}-\eqref{genc} in Proposition~\ref{propgen} requires that
\begin{align}\label{cond}
\big|\eta(\eps,x/N)\big|
=\Big|A_{0}-1+\sum_{n\ge1}A_{n}e^{(1-\lambda_{n}/\lambda_{0})x/N}\Big|
\ll1/N,\quad\text{for }x\ge0.
\end{align}
Notice that the first term in \eqref{cond}, $A_{0}-1$, is independent of $N$. Therefore, if the condition in \eqref{cond} breaks down as $N\to\infty$, then we expect that it is due to the growth of the second term. Hence, we simplify \eqref{cond} to
\begin{align*}
\Big|\sum_{n\ge1}A_{n}e^{(1-\lambda_{n}/\lambda_{0})x/N}\Big|
\ll1/N,\quad\text{for }x\ge0.
\end{align*}
Using the ordering \eqref{order}, we further simplify this to the condition
\begin{align}\label{cond000}
N\exp\Big(\frac{-\lambda_{1}}{\lambda_{0}N}\Big)
\ll1.
\end{align}

Therefore, \eqref{cond000} is a sufficient condition for $T_{N}\approx_{\dist}{\textup{Exp}(\E[\tau]/N)}$. However, to make \eqref{cond000} more readily applicable, we need to estimate $\lambda_{0}$ and $\lambda_{1}$. In the $\eps\ll1$ regime, we have that $\E[\tau]\approx1/\lambda_{0}$ (regardless of $N$). Furthermore, if we have a non-vanishing spectral gap, then $(\lambda_{1}-\lambda_{0})\td\not\to0$ as $\eps\to0$, where $\td$ is the diffusion time, $\td:=L^{2}/D$, where $L>0$ is a characteristic lengthscale describing the size of the domain $\Omega$. Since $(\lambda_{1}-\lambda_{0})\td\le\lambda_{1}\td$, it follows that $\lambda_{1}\td\not\ll1$ for $\eps\ll1$. We thus obtain from \eqref{cond000} the following \emph{sufficient} condition for $T_{N}\approx_{\dist}{\textup{Exp}(\E[\tau]/N)}$, 
\begin{align}\label{cond2}
N\exp\Big(\frac{-\E[\tau]}{N\td}\Big)
\ll1.
\end{align}
Upon noting that the survival probability of the $k$th fastest FPT satisfies
\begin{align*}
\P(T_{k,N}>t)
=\sum_{j=0}^{k-1}{N\choose j}\P(\tau\le t)^{j}\P(\tau>t)^{N-j},
\end{align*}
a similar calculation extends \eqref{cond2} to the general case $k\in\{1,\dots,N\}$. That is, if \eqref{cond2} is satisfied, then this analysis predicts that $T_{k,N}$ is in the regime in \eqref{genc0}-\eqref{genc} in Proposition~\ref{propgen}.

We would now like to derive a \emph{necessary} condition for $T_{N}\approx_{\dist}{\textup{Exp}(\E[\tau]/N)}$. However, we show below that for certain initial searcher distributions, $T_{N}$ is exactly exponentially distributed for all $N\ge1$. Therefore, in order to have a necessary condition for $T_{N}\approx_{\dist}{\textup{Exp}(\E[\tau]/N)}$, we need to restrict to a certain class of initial searcher distributions. Specifically, we suppose that the initial searcher locations cannot be arbitrarily close to the target. More precisely, suppose each searcher has initial distribution given by a probability measure $\mu_{0}$ as in \eqref{initial}, and assume that the support of $\mu_{0}$,
\begin{align*}
U_{0}:=\text{supp}(\mu_{0}),
\end{align*}
does not intersect the closure of the target,
\begin{align}\label{away}
U_{0}\cap \overline{U_{\text{T}}}=\varnothing.
\end{align}
Note that $U_{0}$ is necessarily a closed set.

Assuming \eqref{away}, it was shown in \cite{lawley2020uni} that
\begin{align}\label{prev}
\E[T_{k,N}]
\sim \frac{L^{2}}{4D\ln N}
=\frac{\td}{4\ln N}\quad\text{as }N\to\infty,
\end{align}
where $L>0$ is a certain lengthscale describing the shortest distance a searcher must travel to reach the target. Now, suppose $\td/\E[\tau]\ll1$ (so that a single FPT is in the exponential regime) and note that $\E[T_{N}]$ is a monotonically decreasing function of $N\ge1$. Therefore, for sufficiently small values of $N$ that satisfy \eqref{cond2}, we have that $\E[T_{N}]\approx\E[\tau]/N\gg\td/(4\ln N)$. Then, as $N$ increases, $\E[T_{N}]$ must decrease monotonically to the regime in \eqref{prev}. Since $\E[T_{N}]\approx\E[\tau]/N$ if $T_{N}\approx_{\dist}\text{Exp}(\E[\tau]/N)$, it follows that if $\E[\tau]/N\ll\td/(4\ln N)$, then $T_{N}$ is not in the exponential regime. Put another way, if $T_{N}\approx_{\dist}{\textup{Exp}(\E[\tau]/N)}$, then
\begin{align}\label{condn}
\frac{\E[\tau]}{N}
\not\ll\frac{\td}{4\ln N}.
\end{align}
Hence, \eqref{condn} is a \emph{necessary} condition for $T_{N}\approx_{\dist}\text{Exp}(\E[\tau]/N)$. We emphasize that the condition in \eqref{condn} assumes \eqref{away}. Indeed, we show below that $T_{N}$ can be exactly exponential for all values of $N$ for a certain initial condition which violates \eqref{away}. 

Notice that if we rearrange the necessary condition in \eqref{condn} and take the logarithm of the sufficient condition in \eqref{cond2} and rearrange, then we find that:
\begin{align}\label{suff5}
&\displaystyle\text{If }\frac{\td}{\E[\tau]}
\ll\frac{1}{N\ln N},\text{ then }T_{N}\approx_{\dist}{\textup{Exp}(\E[\tau]/N)}.\\
&\displaystyle\text{If }\frac{\td}{\E[\tau]}
\gg\frac{4\ln N}{N},\text{ then }T_{N}\not\approx_{\dist}{\textup{Exp}(\E[\tau]/N)}.\label{nec5}
\end{align}
Again, \eqref{nec5} assumes \eqref{away}. We illustrate \eqref{suff5}-\eqref{nec5} in Figure~\ref{figcrossover}.

\begin{figure}[t]
\centering
\includegraphics[width=.7\linewidth]{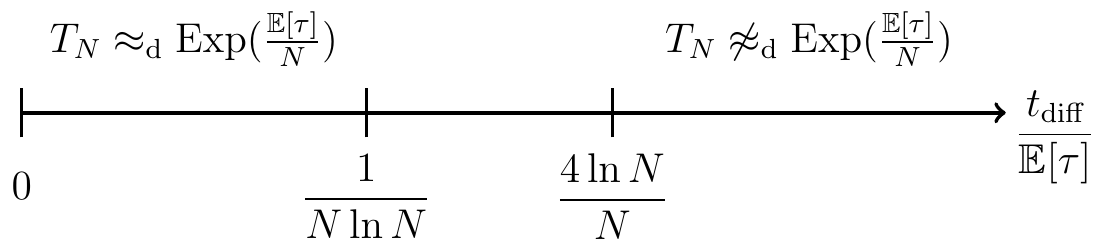}
\caption{
Illustration of the conditions \eqref{suff5}-\eqref{nec5}. The necessary condition in \eqref{nec5} assumes \eqref{away}, which means that the searchers cannot start arbitrarily close to the target. 
}
\label{figcrossover}
\end{figure}

\subsection{Narrow escape with a perfectly absorbing target}\label{narrow}

We now apply the analysis of the previous sections to some prototypical examples. We first consider a diffusive searcher in a bounded domain with small targets, which is the narrow escape problem \cite{holcman2014}. In particular, consider the setup of Section~\ref{spectral} in dimension $d\in\{2,3\}$ with pure diffusion (i.e.\ $V\equiv0$ in \eqref{sde}). In this case, the natural small parameter is the dimensionless target size, $\eps:=\sigma$, where
\begin{align}\label{sigma}
\sigma
:=\Big(\frac{|\partial\Omega_{\text{T}}|}{|\partial\Omega|}\Big)^{1/(d-1)}
\ll1,
\end{align}
which compares the $(d-1)$-dimensional area of the target $\partial\Omega_{\text{T}}$ to the $(d-1)$-dimensional area of the rest of the boundary $\partial\Omega$. As a technical condition, assume that the isoperimetric ratio remains bounded,
\begin{align}\label{iso}
\frac{|\partial\Omega|^{1/(d-1)}}{|\Omega|^{1/d}}
=\O(1)\quad\text{for }\sigma\ll1,
\end{align}
where $|\Omega|$ denotes the $d$-dimensional volume of the domain (\eqref{iso} prevents pathological cases \cite{holcman2014}).

Then in the $\sigma\to0$ limit (i.e.\ the small target or narrow escape limit), it is well-known \cite{holcman2014} that $\tau$ becomes exponentially distributed with a vanishing rate $\lambda_{0}>0$, where the asymptotic form of $\lambda_{0}$ depends on the dimension $d\in\{2,3\}$ and the geometry of the domain and the target. The basic idea is that in the limit $\sigma\to0$, the entire boundary becomes reflecting and the spectral problem in \eqref{evalproblem} approaches the Neumann spectral problem,
\begin{align*}
-D\Delta u_{n}^{\neu}
&=\lambda_{n}^{\neu}u_{n}^{\neu},\quad x\in\Omega,\\
\tfrac{\partial}{\partial\n}u_{n}^{\neu}
&=0,\quad x\in\partial\Omega.
\end{align*}
In particular, in the limit $\sigma\to0$, we have that
\begin{align*}
\lambda_{0}
&\to\lambda_{0}^{\neu}=0,\quad \lambda_{n}\to\lambda_{n}^{\neu}>0\quad n\ge1,\\
u_{0}
&\to u_{0}^{\neu}=1,\quad u_{n}\to u_{n}^{\neu}\quad n\ge1,
\end{align*}
and the orthonormality (see \eqref{orthonormal} with $\rho\equiv1/|\Omega|$ in \eqref{rho}) implies that
\begin{align*}
(u_{0},1)_{\rho}\to(u_{0}^{\neu},1)_{\rho}=1,\quad
(u_{n},1)_{\rho}\to(u_{n}^{\neu},1)_{\rho}=0\quad n\ge1.
\end{align*}
Therefore, $A_{n}\to\delta_{n0}\in\{0,1\}$ as $\sigma\to0$.

To illustrate, if the target is the union of $d$-dimensional spheres of radius $r>0$ centered at $M\ge1$ distinct points $z_{1},\dots,z_{M}\in\R^{d}$,
\begin{align*}
\partial\Omega_{\text{T}}
:=\cup_{m=1}^{M}\{x\in\R^{d}:\|x-z_{m}\|=\sigma r\},
\end{align*}
then the principal eigenvalue has the asymptotic behavior \cite{kolokolnikov2005, cheviakov11},
\begin{align*}
\lambda_{0}
\sim
\begin{cases}
-2\pi DM/(|\Omega|\log\sigma) & \text{if }d=2,\\
4\pi DMr\sigma/|\Omega| & \text{if }d=3,
\end{cases}
\quad\text{as }\sigma\to0,
\end{align*}
Hence, the diverging MFPT of a single searcher satisfies
\begin{align*}
Md(\td)^{-1}\E[\tau]
\sim
\begin{cases}
-\ln \sigma & \text{if }d=2,\\
\sigma^{-1} & \text{if }d=3,
\end{cases}
\quad\text{as }\sigma\to0,
\end{align*}
if we define the characteristic diffusion timescale,
\begin{align*}
\td
:=\begin{cases}
|\Omega|/(\pi D) & \text {if }d=2,\\
|\Omega|/(\tfrac{4}{3}\pi rD) & \text{if }d=3.
\end{cases}
\end{align*}
The sufficient condition \eqref{cond2} thus becomes
\begin{align*}
1
&\gg
\begin{cases}
N\sigma^{1/(2MN)} & \text{if }d=2,\\
N\exp\big(-(3MN\sigma)^{-1}\big)  & \text{if }d=3.
\end{cases}
\end{align*}

\subsection{Narrow escape and/or small target reactivity for partial absorption}\label{proto2}

The analysis in Section~\ref{narrow} above quickly extends to the case of partially absorbing targets, assuming the targets are small and/or have low reactivity. Let $\Omega\subset\R^{d}$ be as in Section~\ref{spectral} in dimension $d\ge1$, let $\sigma$ be as in \eqref{sigma} above (if $d=1$, then we set $\sigma=0$ and define $\sigma^{d-1}=1$), and assume \eqref{iso} if $d\ge2$. In this case of an imperfect target, the survival probability again satisfies \eqref{backward}, except the absorbing boundary condition on the target is replaced by the Robin boundary condition,
\begin{align*}
D\tfrac{\partial}{\partial \n}\S
=\kon S,\quad x\in\partial\Omega_{\text{T}},
\end{align*}
where $\kon\in(0,\infty)$ is a parameter describing the reactivity of the target \cite{grebenkov2006}.

Define the dimensionless reactivity,
\begin{align}\label{kappa}
\kappa
:=\frac{\kon L}{D},
\end{align}
for some lengthscale $L>0$. In the limit that the target is small and/or not reactive,
\begin{align*}
\eps:=
\kappa\sigma^{d-1}\to0,
\end{align*}
the survival probability problem \eqref{backward} again approaches the Neumann problem and the analysis in Section~\ref{narrow} applies with the vanishing principal eigenvalue satisfying \cite{lawley2019imp}
\begin{align}\label{pr}
\lambda_{0}
\sim\frac{D}{LL_{0}}\kappa\sigma^{d-1}\quad\text{as }\kappa\sigma^{d-1}\to0,
\end{align}
where $L_{0}=|\Omega|/|\partial\Omega|>0$ is the lengthscale describing the $d$-dimensional volume of the domain to the $(d-1)$-dimensional area of the boundary (if $d=1$, then $L=L_{0}$ is the length of the interval $\Omega$ and we take $\sigma^{d-1}=1$).

Hence, the diverging MFPT of a single searcher satisfies
\begin{align*}
d(\td)^{-1}\E[\tau]
\sim
\frac{1}{\kappa\sigma^{d-1}},\quad\text{as }\kappa\sigma^{d-1}\to0,
\end{align*}
if we define the diffusion time, $\td:=\frac{dLL_{0}}{D}$. The sufficient condition \eqref{cond2} thus becomes
\begin{align*}
1
&\gg
N\exp\big(-(dN\kappa\sigma^{d-1})^{-1}\big).
\end{align*}

\subsection{Escape from a potential well}\label{ou}

In Sections~\ref{narrow} and \ref{proto2} above, the FPT was slow because the target was small and/or the target had a small reactivity. In this subsection, we consider the case that the FPT is slow because the searcher must escape a deep potential well to reach the target.

It is well-known that the Brownian escape time from a potential becomes exponentially distributed with vanishing rate as the potential depth grows \cite{shenoy1984, hanggi1990}. To make the calculations explicit, we consider a quadratic potential, so that $X(t)$ is a $d$-dimensional Ornstein-Uhlenbeck process. Specifically, let $X(t)$ be as in \eqref{sde} in Section~\ref{spectral} where $V:\R^{d}\to\R$ is the quadratic potential,
\begin{align*}
V(x)
:=\frac{\theta}{2}\|x\|^{2},
\end{align*}
where $\|\cdot\|$ denotes the $d$-dimensional Euclidean norm and $\theta>0$ is some positive parameter. Let $\Omega=\{x\in\R^{d}:\|x\|<L\}$ be the $d$-dimensional ball of radius $L>0$ and let the target be the entire boundary $\partial\Omega_{\text{T}}=\partial\Omega=\{x\in\R^{d}:\|x\|=L\}$ (see Figure~\ref{figschem}b). The FPT $\tau>0$ in \eqref{tau} is then
\begin{align*}
\tau
:=\inf\{t>0:\|X(t)\|>L\}.
\end{align*}

In this case, the survival probability can be written in the form \eqref{form} (see equation (67) in \cite{grebenkov2014}) and the FPT becomes exponential in the limit of a deep potential. Specifically, define the dimensionless parameter,
\begin{align*}
\eps
:=\frac{2D}{\theta L^{2}}>0,
\end{align*}
which measures the noise strength $D$ to the potential depth $\theta$ and the escape radius $L>0$. In the limit $\eps\to0$, the escape time $\tau$ is exponentially distributed with rate \cite{grebenkov2014}
\begin{align*}
\lambda_{0}
\sim\frac{4e^{-1/\eps}}{\td\Gamma(d/2)\eps^{d/2+1}}\quad\text{as }\eps\to0,
\end{align*}
where $\td:=L^{2}/D$ and $\Gamma(\cdot)$ denotes the gamma function. Further, the larger eigenvalues diverge as $\lambda_{n}\sim4n(\eps\td)^{-1}$ as $\eps\to0$ for $n\ge1$ \cite{grebenkov2014}.

Hence, the diverging MFPT of a single searcher satisfies
\begin{align*}
(\td)^{-1}\E[\tau]
\sim\frac{\Gamma(d/2)}{4}\eps^{d/2+1}e^{1/\eps}\quad\text{as }\eps\to0.
\end{align*}
The sufficient condition \eqref{cond2} thus becomes
\begin{align*}
1
&\gg
N\exp\big(-\frac{\Gamma(d/2)}{4N}\eps^{d/2+1}e^{1/\eps}\big).
\end{align*}

\subsection{Eigenfunction initial condition}

In Sections~\ref{narrow}-\ref{ou}, the FPT was approximately exponential because the principal eigenvalue $\lambda_{0}$ was much smaller than the diffusion rate $1/\td$. A simple situation in which the FPT is exactly exponential is if the initial searcher distribution is the so-called quasi-stationary distribution \cite{meleard12}. In particular, consider the setup of Section~\ref{spectral} and suppose the distribution of $X(0)$ is given by the product of the weight function in \eqref{rho} and the principal eigenfunction in \eqref{evalproblem},
\begin{align*}
\P(X(0)\in B)
=\frac{1}{\mathcal{N}}\int_{B}u_{0}(x)\rho(x)\,\dd x,
\quad B\subset\Omega,\quad 
\text{where }\mathcal{N}
:=\int_{\Omega}u_{0}(x)\rho(x)\,\dd x.
\end{align*}
That is, suppose the initial searcher position has the probability density function $u_{0}(x)\rho(x)/\mathcal{N}$. Hence, the orthonormality in \eqref{orthonormal} ensures that $A_{n}=\delta_{0n}$ (see \eqref{As}), and thus the series \eqref{form} collapses to 
\begin{align*}
S(t)=e^{-\lambda_{0}t},\quad t\ge0,
\end{align*}
which means that $\tau$ is exactly exponential with mean $\E[\tau]=1/\lambda_{0}$. Hence, the distribution of $T_{k,N}$ is exactly given by the distributions \eqref{genc0}-\eqref{genc} in Proposition~\ref{propgen} for all $N\ge1$. In particular, $T_{N}=_{\dist}\textup{Exp}(\E[\tau]/N)$ for all $N\ge1$, and thus, for example,
\begin{align*}
\E[T_{N}]
=\frac{1}{N\lambda_{0}}
=\frac{\E[\tau]}{N},\quad\text{for all }N\ge1.
\end{align*}

This illustrates that the behavior $\E[T_{N}]\sim\td(4\ln N)^{-1}$ for large $N$ may not hold if \eqref{away} is violated. In fact, we find below that $\E[T_{N}]$ may decay like $N^{-1}$ or $N^{-2}$ for other choices of initial conditions.

\section{Fast escape regime}\label{fast}

In Section~\ref{slow} above, we found that the fastest FPT, $T_{N}$, is approximately exponential if the number of searchers $N$ is sufficiently small, and quantified ``sufficiently small'' in terms of the ratio of the MFPT of a single searcher to a characteristic diffusion timescale. The next natural question is what happens to the distribution of the fastest FPTs in the limit $N\to\infty$.

In this section, we show how to go from the short time behavior of the survival probability of a single FPT, $S(t):=\P(\tau>t)$, to the distribution of $T_{k,N}$ in the limit $N\to\infty$. In particular, assuming that $S(t)$ has the following short time behavior,
\begin{align}\label{sapc}
\P(\tau\le t)
=1-S(t)
\sim At^{p}e^{-C/t}\quad\text{as }t\to0+
\end{align}
for some constants $A>0$, $p\in\R$, and $C\ge0$, we find the distribution and all the moments of $T_{k,N}$ in the limit $N\to\infty$ (the case $C>0$ was handled in \cite{lawley2020dist}). The proofs of the results of this section are collected in the Appendix. 

The short time behavior in \eqref{sapc} holds in many diverse scenarios (see below and also see the Discussion section in \cite{lawley2020dist}). In particular, if the searchers cannot start arbitrarily close to the target (see \eqref{away}), then one typically has that $C=\td/4:=L^{2}/(4D)$, where $D>0$ is the searcher diffusivity and $L>0$ is the shortest distance from the searcher starting locations to the target (this holds for free Brownian motion, and in fact much more general diffusive processes). The parameters $A>0$ and $p\in\R$ in \eqref{sapc} depend on more details in the problem, such as spatial dimension, target size, target reactivity, etc.\ (see, for example, section~\ref{competition} below). However, if the searchers can start arbitrarily close to the target, then we find below that \eqref{sapc} can hold with $C=0$ and $p>0$.


In the results below, the limiting distribution of $T_{k,N}$ is described in terms of Gumbel, Weibull, and generalized Gamma distributions. For convenience, we first give the definitions of these distributions.

\begin{definition*}
A random variable $X\ge0$ has a \emph{Weibull distribution} with scale parameter $t>0$ and shape parameter $p>0$ if
\begin{align}\label{xweibull}
\P(X>x)
=\exp(-(x/t)^{p}),\quad x\ge0.
\end{align}
If \eqref{xweibull} holds, then we write
\begin{align*}
X=_{\textup{d}}\textup{Weibull}(t,p).
\end{align*}
Notice that if \eqref{xweibull} holds with $p=1$, then $X=_{\textup{d}}\textup{Exp}(t)$.

A random variable $X\ge0$ has a \emph{generalized Gamma distribution} with parameters $t>0$, $p>0$, $k>0$ if
\begin{align}\label{gGamma}
\P(X>x)
=\frac{\Gamma(k,(x/t)^{p})}{\Gamma(k)},\quad x\ge0,
\end{align}
where $\Gamma(a,z):=\int_{z}^{\infty}u^{a-1}e^{-u}\,\dd u$ denotes the upper incomplete gamma function. If \eqref{gGamma} holds, then we write
\begin{align*}
X=_{\textup{d}}\textup{gen}\Gamma(t,p,k).
\end{align*}

A random variable $X$ has a \emph{Gumbel distribution} with location parameter $b\in\R$ and scale parameter $a>0$ if\footnote{Some authors define a Gumbel distribution slightly differently, by saying that $-X$ has a Gumbel distribution with shape $-b$ and scale $a$ if \eqref{xgumbel} holds.}
\begin{align}\label{xgumbel}
\P(X>x)
=\exp\Big[-\exp\Big(\frac{x-b}{a}\Big)\Big],\quad\text{for all } x\in\R.
\end{align}
If \eqref{xgumbel} holds, then we write
\begin{align*}
X=_{\textup{d}}\textup{Gumbel}(b,a).
\end{align*}

\end{definition*}


\subsection{The case $C=0$}

If the initial searcher distribution is such that the searchers can start arbitrarily close to the target (meaning \eqref{away} is violated), then the behavior of the survival probability in \eqref{sapc} can hold with $C=0$ and $p>0$, which yields a drastically different distribution of the fastest FPTs compared to the case $C>0$.

The first result below gives the full distribution of $T_{N}$ for large $N$ assuming $C=0$ in \eqref{sapc}. Throughout this work, ``$f\sim g$'' means $f/g\to1$.

\begin{theorem}\label{main}
Let $\{\tau_{n}\}_{n\ge1}$ be an iid sequence of random variables and assume that for some ${{A}}>0$ and $p>0$, we have that
\begin{align}\label{short}
\P(\tau_{n}\le t)
&\sim{{A}} t^{p}\quad\text{as }t\to0+.
\end{align}
The following rescaling of $T_{N}:=\min\{\tau_{1},\dots,\tau_{N}\}$ converges in distribution to a Weibull random variable,
\begin{align*}
(AN)^{1/p}T_{N}
\to_{\dist}
\textup{Weibull}(1,p)\quad\text{as }N\to\infty.
\end{align*}
\end{theorem}

Roughly speaking, Theorem~\ref{main} means that the distribution of $T_{N}$ is
\begin{align*}
T_{N}\approx_{\dist}\textup{Weibull}\big((AN)^{-1/p},p\big)\quad\text{for $N$ sufficiently large}.
\end{align*}
The next result approximates all the moments of the fastest FPT.

\begin{theorem}\label{moments}
Under the assumptions of Theorem~\ref{main}, suppose further that
\begin{align*}
\E[T_{N}]<\infty\quad\text{for some }N\ge1.
\end{align*}
Then for each moment $m\in(0,\infty)$, we have that
\begin{align*}
\E[(T_{N})^{m}]
&\sim \frac{\Gamma(1+m/p)}{(AN)^{m/p}}\quad\text{as }N\to\infty.
\end{align*}
Hence,
\begin{align*}
\E[T_{N}]
&\sim \frac{\Gamma(1+1/p)}{(AN)^{1/p}}\quad\text{as }N\to\infty,\\
\textup{Variance}(T_{N})
&\sim\frac{\Gamma(1+2/p)
-\big(\Gamma(1+1/p)\big)^{2}}{(AN)^{2/p}}
\quad\text{as }N\to\infty.
\end{align*}
\end{theorem}

\begin{remark}\label{remarksuff1}
We are interested in estimating when a particular system is in the large $N$ regime of Theorems~\ref{main} and \ref{moments}. By Theorem~\ref{moments}, we are assured that a system is in this large $N$ regime, at least for all moments $m\ge1$, if
\begin{align}\label{condgood}
\frac{1}{(AN)^{1/p}}\ll1.
\end{align}
\end{remark}

We now generalize Theorem~\ref{main} on the fastest FPT to the $k$th fastest FPT,
\begin{align}\label{tkn}
T_{k,N}
:=\min\big\{\{\tau_{1},\dots,\tau_{N}\}\backslash\cup_{j=1}^{k-1}\{T_{j,N}\}\big\},\quad k\in\{1,\dots,N\},
\end{align}
where $T_{1,N}:=T_{N}$. Indeed, some physical scenarios depend not on the fastest FPT, but rather the $k$th fastest FPT for $k\ge2$ \cite{schuss2019} (see \cite{basnayake2019fast} for the case $k=2$ for calcium-induced calcium release in dendritic spines). The following theorem gives the distribution of $T_{k,N}$ for large $N$. 

\begin{theorem}\label{kth}
Under the assumption of Theorem~\ref{main}, the following rescaling of $T_{k,N}$ in \eqref{tkn} converges in distribution to a generalized Gamma random variable,
\begin{align*}
(AN)^{1/p}T_{k,N}
\to_{\dist}
\textup{gen}\Gamma(1,p,k)\quad\text{as }N\to\infty.
\end{align*}
\end{theorem}

Roughly speaking, Theorem~\ref{kth} means that the distribution of $T_{k,N}$ is
\begin{align*}
T_{k,N}
\approx_{\dist}
\textup{gen}\Gamma\big((AN)^{-1/p},p,k\big)\quad\text{for $N$ sufficiently large}.
\end{align*}
The next result approximates all the moments of the $k$th fastest FPT as $N\to\infty$.

\begin{theorem}\label{kth moment}
Under the assumptions of Theorem~\ref{moments}, for each moment $m\in(0,\infty)$, we have that
\begin{align*}
\E[(T_{k,N})^{m}]
&\sim\frac{\Gamma(k+m/p)/\Gamma(k)}{(AN)^{m/p}}
\quad\text{as }N\to\infty.
\end{align*}
Hence,
\begin{align*}
\E[T_{k,N}]
&\sim\frac{\Gamma(k+1/p)/\Gamma(k)}{(AN)^{1/p}}
\quad\text{as }N\to\infty,\\
\textup{Variance}(T_{k,N})
&\sim\frac{\Gamma(k+2/p)/\Gamma(k)
-\big(\Gamma(k+1/p)/\Gamma(k)\big)^{2}}{(AN)^{2/p}}
\quad\text{as }N\to\infty.
\end{align*}
\end{theorem}

\subsection{The case $C>0$}

The case that $C>0$ in \eqref{sapc} was handled in \cite{lawley2020dist} and characterizes the case that the searchers cannot start arbitrarily close to the target (see \eqref{away}). For convenience, we repeat the result here in the case $k=1$ (the case of a general $k$ was also handled in \cite{lawley2020dist} but we omit it for brevity).

\begin{theorem}\label{repeat}[Proven in Reference~\cite{lawley2020dist}]
Let $\{\tau_{n}\}_{n\ge1}$ be an iid sequence of nonnegative random variables, and assume that there exists constants $C>0$, $A>0$, and $p\in\R$ so that
\begin{align*}
\P(\tau_{1}\le t)
\sim At^{p}e^{-C/t}\quad\text{as }t\to0+.
\end{align*}
The following rescaling of $T_{N}:=\min\{\tau_{1},\dots,\tau_{N}\}$ converges in distribution to a Gumbel random variable,
\begin{align}\label{cd}
\frac{T_{N}-b_{N}}{a_{N}}
\to_{\textup{d}}
X=_{\textup{d}}\textup{Gumbel}(0,1)\quad\text{as }N\to\infty,
\end{align}
where 
\begin{align}\label{ababuf}
\begin{split}
a_{N}
&=\frac{C}{(\ln N)^{2}},\quad
b_{N}
=\frac{C}{\ln N}
+\frac{Cp\ln(\ln(N))}{(\ln N)^{2}}
-\frac{C\ln(AC^{p})}{(\ln N)^{2}}.
\end{split}
\end{align}
\end{theorem}

The next result allows approximates all the moments of the fastest FPT.

\begin{theorem}\label{repeat2}[Proven in Reference~\cite{lawley2020dist}]
Under the assumptions of Theorem~\ref{repeat}, suppose further that $\E[T_{N}]<\infty$ for some $N\ge1$. Then for each moment $m\in(0,\infty)$, we have that
\begin{align*}
\E[(T_{N}-b_{N})^{m}]
\sim a_{N}^{m}\E[X^{m}]\quad\text{as }N\to\infty,
\end{align*}
where $X=_{\dist}\textup{Gumbel}(0,1)$. In particular,
\begin{align}
\E[T_{N}]
&=b_{N}-\gamma a_{N}+o(a_{N})\nonumber\\
&=\frac{C}{\ln N}
\Big[1
+\frac{p\ln(\ln(N))}{\ln N}
-\frac{\ln(AC^{p})+\gamma}{\ln N}
+o\big(1/\ln N\big)\Big],\label{meanlog}\\
\textup{Variance}(T_{N})
&\sim\frac{\pi^{2}}{6}a_{N}^{2}
=\frac{\pi^{2}}{6}\frac{C^{2}}{(\ln N)^{4}}\quad\text{as }N\to\infty,\nonumber
\end{align}
where $\gamma\approx0.5772$ is the Euler-Mascheroni constant.
\end{theorem}

\begin{remark}\label{remarksuff2}
We are interested in estimating when a particular system is in the large $N$ regime of Theorems~\ref{repeat} and \ref{repeat2}. By Theorem~\ref{repeat2}, we are assured that a system is in this large $N$ regime, at least for all moments $m\ge1$, if
\begin{align*}
\Big|\frac{\ln(AC^{p})+\gamma}{\ln N}\Big|\ll1.
\end{align*}
\end{remark}

\section{Competition between slow and fast escape}\label{competition}

In this section, we use the results of Sections~\ref{slow} and \ref{fast} to investigate (i) the competition between slow and fast escape ($\eps\to0$ versus $N\to\infty)$, (ii) the effects of initial conditions, and (iii) the effects of target reactivity. Consider the $d$-dimensional annular domain,
\begin{align*}
\Omega
:=\{x\in\R^{d}:a<\|x\|<{{R}}\},\quad d\in\{1,2,3\},
\end{align*}
where $\|\cdot\|$ denotes the Euclidean length. We take $a=0$ in dimension $d=1$ and $a\in(0,R)$ in dimensions $d\in\{2,3\}$. The boundary $\partial\Omega=\partial\Omega_{\text{T}}\cup\partial\Omega_{\text{R}}$ consists of the target at the inner boundary, $\partial\Omega_{\text{T}}:=\{x\in\R^{d}:\|x\|=a\}$, and the outer boundary, $\partial\Omega_{\text{R}}:=\{x\in\R^{d}:\|x\|={{R}}\}$. Let $\{X(t)\}_{t\ge0}$ denote the path of a searcher with diffusivity $D>0$ diffusing in $\Omega$ with reflecting boundary conditions on $\partial\Omega$.

Due to the symmetry in the problem, the survival probability,
\begin{align}\label{sr}
\S(r,t)
:=\P(\tau>t\,|\,\|X(0)\|=r),
\end{align}
satisfies the backward Fokker-Planck (backward Kolmogorov) equation,
\begin{align}\label{sann}
\begin{split}
\tfrac{\partial}{\partial t}\S
&=D\big(\tfrac{d-1}{r}\tfrac{\partial}{\partial r}+\tfrac{\partial^{2}}{\partial r^{2}}\big)\S,\quad t>0,\,r\in(a,{{R}}),\\
\tfrac{\partial}{\partial r}\S
&=0,\quad r={{R}},\\
\S
&=1,\quad t=0,
\end{split}
\end{align}
with either an absorbing Dirichlet condition or a Robin condition at the target. We write these two cases in a single boundary condition,
\begin{align}
D\tfrac{\partial}{\partial r}\S
&=\kon \S,\quad r=a,\label{bcimp}
\end{align}
where $\kon\in(0,\infty)$ corresponds to a partially absorbing target and $\kon=\infty$ corresponds to a perfectly absorbing target and \eqref{bcimp} means $\S=0$ at $r=a$. It is convenient to define the characteristic lengthscale and diffusion time,
\begin{align*}
L
&:=R-a>0,\\
\td
&:=\frac{L^{2}}{D}>0,
\end{align*}
and the dimensionless target size and target reactivity,
\begin{align}
\sigma
&:=\frac{a}{R}\in[0,1),\nonumber\\
\kappa
&:=\frac{\kon L}{D}>0,\label{kappaconv}
\end{align}
where $\kappa=\infty$ corresponds to a perfectly absorbing target. We now analyze the fastest FPT in four cases, depending on the initial searcher distribution and whether the target is perfectly or partially absorbing.

\subsection{Case 1: $\|X(0)\|=R$ and $\kappa=\infty$}\label{section perfect}

First consider the case that $\|X(0)\|=R$ (a Dirac delta function initial condition) and a perfectly absorbing target ($\kappa=\infty$). In the small target case ($\eps:=\sigma\ll1$ in dimensions $d\in\{2,3\}$), the FPT is approximately exponentially distributed and its diverging mean satisfies
\begin{align}\label{mean00}
d(\td)^{-1}\E[\tau]
\sim\begin{cases}
-\ln\sigma & \text{if }d=2,\\
\sigma^{-1} & \text{if }d=3,
\end{cases}
\quad\text{as }\sigma\to0.
\end{align}
Hence, for any fixed $N\ge1$, the distribution of $T_{k,N}$ is given by Proposition~\ref{propgen} in the limit $\sigma\to0$. In particular, we have that
\begin{align}\label{dist1}
T_{N}
\approx_{\dist}{\textup{Exp}(\E[\tau]/N)},
\end{align}
and thus
\begin{align}\label{mean1}
\E[T_{N}]\approx\frac{\E[\tau]}{N}.
\end{align}

On the other hand, if we fix any value of $\sigma\in(0,1)$, and take $N\to\infty$, then the distribution of $T_{k,N}$ is completely determined by the short-time behavior of the survival probability $S(t)=\S(R,t)$. In the Appendix, we show that $S(t)$ has the following short-time behavior,
\begin{align}\label{short1}
1-S(t)
&\sim At^{p}e^{-C/t}\quad\text{as }t\to0+,
\end{align}
where
\begin{align*}
A
=\frac{2}{\sqrt{\pi}}\sqrt{\frac{1}{\td}}\sigma^{(d-1)/2},\quad
p
=\frac{1}{2},\quad
C
=\frac{\td}{4}>0.
\end{align*}
Hence, for any fixed $\sigma\in(0,1)$, the distribution of $T_{k,N}$ is given by Theorem~\ref{repeat} in the limit $N\to\infty$. In particular,
\begin{align}\label{dist2}
T_{N}
\approx_{\dist}\textup{Gumbel}\left(\frac{C}{\ln N}
+\frac{Cp\ln(\ln(N))}{(\ln N)^{2}}
-\frac{C\ln(AC^{p})}{(\ln N)^{2}}\,,\,\frac{C}{(\ln N)^{2}}\right),
\end{align}
and thus by Theorem~\ref{repeat2},
\begin{align}\label{mean2}
\E[T_{N}]
&=\frac{C}{\ln N}
\Big[1
+\frac{p\ln(\ln(N))}{\ln N}
-\frac{\ln(AC^{p})+\gamma}{\ln N}
+o\big(1/\ln N\big)\Big],
\end{align}
as $N\to\infty$, where $\gamma\approx0.5772$ is the Euler-Mascheroni constant.

\subsection{Case 2: $\|X(0)\|=R$ and $\kappa<\infty$}\label{section imperfect}

Suppose again that $\|X(0)\|=R$, but now suppose that the target is partially absorbing ($\kappa<\infty$). In the case that the target is small and/or has small reactivity ($\eps:=\kappa\sigma^{d-1}\ll1$), the FPT is approximately exponentially distributed and its diverging mean satisfies
\begin{align}\label{mean01}
d(\td)^{-1}\E[\tau]
\sim
\frac{1-\sigma^{d}}{(1-\sigma)^{2}}\frac{1}{\kappa\sigma^{d-1}}
\quad\text{as }\kappa\sigma^{d-1}\to0.
\end{align}
Hence, for any fixed $N\ge1$, the distribution of $T_{k,N}$ is given by Proposition~\ref{propgen} in the limit $\kappa\sigma^{d-1}\to0$. In particular, \eqref{dist1} and \eqref{mean1} hold.

On the other hand, if we fix any value of $\kappa\sigma^{d-1}>0$ and take $N\to\infty$, then the distribution of $T_{k,N}$ is completely determined by the short-time behavior of the survival probability $S(t)=\S(R,t)$. In the Appendix, we show that $S(t)$ has the short-time behavior in \eqref{short1}, 
where
\begin{align}\label{apc3}
A
=\frac{4}{\sqrt{\pi}}(\td)^{-3/2}\kappa\sigma^{(d-1)/2},\quad
p=\frac{3}{2},\quad
C=\frac{\td}{4}>0.
\end{align}
Hence, for any fixed $\kappa\sigma^{d-1}$, the distribution of $T_{k,N}$ is given by Theorem~\ref{repeat} in the limit $N\to\infty$. In particular, $T_{N}$ satisfies \eqref{dist2} and \eqref{mean2} with $A$, $p$, and $C$ in \eqref{apc3}.

\subsection{Case 3: $X(0)=_{\dist}\text{Uniform}(\Omega)$ and $\kappa=\infty$}\label{section uniform perfect}

Suppose the target is perfectly absorbing as in \eqref{section perfect} (i.e.\ $\kappa=\infty$), but now suppose that each searcher is initially uniformly distributed in the domain $\Omega$. Hence, the survival probability of a single FPT is
\begin{align}\label{integrate0}
S^{}(t)
:=\P(\tau>0\,|\,X(0)=_{\dist}\textup{Uniform}(\Omega))
=\int_{a}^{R}\S(r,t)\,\frac{d}{R^{d}-a^{d}}r^{d-1}\dd r.
\end{align}
Note that \eqref{integrate0} is the definition of the survival probability of a single searcher in the case that the searcher is initially uniformly distributed in the domain.

In the case of a small target ($\eps:=\sigma\ll1$ in dimensions $d\in\{2,3\}$), the FPT $\tau$ and the extreme $T_{k,N}$ are as in Section~\ref{section perfect} above. Similar to Section~\ref{section perfect}, if we fix any value of $\sigma\in(0,1)$, and take $N\to\infty$, then the distribution of $T_{k,N}$ is again completely determined by the short-time behavior of the survival probability $S^{}(t)$. However, in the case of a uniform initial distribution, the short time behavior of the survival probability is fundamentally different than in Section~\ref{section perfect}. In the Appendix, we show that $S^{}(t)$ has the following short-time behavior,
\begin{align}\label{short2}
1-S^{}(t)
&\sim At^{p}\quad\text{as }t\to0+,
\end{align}
where
\begin{align}\label{apc333}
A
=\frac{2d(1-\sigma)}{\sqrt{\pi}(1-\sigma^{d})}\frac{\sigma^{d-1}}{\sqrt{\td}},\quad
p=\frac{1}{2}.
\end{align}
Hence, for any fixed $\sigma\in(0,1)$, the distribution of $T_{k,N}$ is given by Theorem~\ref{main} in the limit $N\to\infty$. In particular,
\begin{align}\label{dist3}
(AN)^{2}T_{N}
\to_{\dist}\textup{Weibull}(1,1/2)\quad\text{as }N\to\infty,
\end{align}
and thus by Theorem~\ref{moments},
\begin{align}\label{mean3}
\E[T_{N}]
&\sim
\frac{2}{A^{2}N^{2}}
=\td\frac{\pi(1-\sigma^{d})^{2}}{2d^{2}(1-\sigma)^{2}\sigma^{2(d-1)}}\frac{1}{N^{2}}\quad\text{as }N\to\infty.
\end{align}
Notice that if $\sigma\ll1$, then \eqref{mean3} means that (using \eqref{mean00})
\begin{align*}
\E[T_{N}]/\td
\approx
\begin{cases}
\frac{\pi}{8}\exp\big(4\E[\tau]/\td\big)\frac{1}{N^{2}} & \text{if }d=2,\\
\frac{9\pi}{2}(\E[\tau]/\td)^{4}\frac{1}{N^{2}} & \text{if }d=3,
\end{cases}\quad\text{for $N$ sufficiently large}.
\end{align*}

\subsection{Case 4: $X(0)=_{\dist}\text{Uniform}(\Omega)$ and $\kappa<\infty$}\label{section uniform imperfect}

Finally, suppose the searchers are initially uniformly distributed in the domain $\Omega$ and the target is partially absorbing ($\kappa<\infty$). Hence, the survival probability of a single FPT is obtained by integrating the survival probability $\S(r,t)$ from Case 2 in Section~\ref{section imperfect} as in \eqref{integrate0}.

In the case of a small target and/or low reactivity ($\eps:=\kappa\sigma^{d-1}\ll1$), the FPT $\tau$ and the extreme $T_{k,N}$ are as in Section~\ref{section imperfect} above. Similar to Section~\ref{section perfect}, if we fix any value of $\kappa\sigma^{d-1}\in(0,1)$, and take $N\to\infty$, then the distribution of $T_{k,N}$ is again completely determined by the short-time behavior of the survival probability $S_{}^{}(t)$. In the Appendix, we show that $S_{}^{}(t)$ has the short-time behavior in \eqref{short2}, where
\begin{align}\label{ap4}
A
=\frac{d(1-\sigma)}{1-\sigma^{d}}\frac{\kappa\sigma^{d-1}}{\td},\quad
p=1.
\end{align}
Hence, for any fixed $\kappa\sigma^{d-1}$, the distribution of $T_{k,N}$ is given by Theorem~\ref{main} in the limit $N\to\infty$. In particular, since $p=1$ in \eqref{ap4}, $T_{N}$ is asymptotically exponentially distributed for large $N$,
\begin{align*}
ANT_{N}
\to_{\dist}\textup{Weibull}(1,1)
=_{\dist}\text{Exp}(1)\quad\text{as }N\to\infty.
\end{align*}
Further, Theorem~\ref{moments} implies that
\begin{align}\label{meanan}
\E[T_{N}]
\sim
\frac{1}{AN}
=\td\frac{1-\sigma^{d}}{d(1-\sigma)\kappa\sigma^{d-1}}\frac{1}{N}\quad\text{as }N\to\infty,
\end{align}
Upon using \eqref{mean01}, if $\sigma\ll1$, then \eqref{meanan} implies
\begin{align*}
\E[T_{N}]
\approx\frac{\E[\tau]}{N}\quad\text{for $N$ sufficiently large},
\end{align*}
and thus we conclude that $T_{N}$ is approximately exponential with mean $\E[\tau]/N$ for both large $N$ and small $N$.

\subsection{Comparison of 4 cases}\label{section compare 4}

If we apply the condition in \eqref{cond2} to the 4 cases above, we find that a sufficient condition for $T_{N}\approx_{\dist}{\textup{Exp}(\E[\tau]/N)}$ is
\begin{align}\label{theta1}
1
&\gg
\theta_{\text{exp}}:=
\begin{cases}
N\sigma^{1/(2N)} & \text{if }d=2,\,\kappa=\infty,\\
N\exp\big(-(3N\sigma)^{-1}\big)  & \text{if }d=3,\,\kappa=\infty,\\
N\exp\big(-(dN\kappa\sigma^{d-1})^{-1}\big) & \text{if }d\in\{1,2,3\},\,\kappa<\infty.
\end{cases}
\end{align}
Further, if $\|X(0)\|=R$, then we apply the condition in Remark~\ref{remarksuff2} to Cases 1 and 2 above to find that a sufficient condition for $T_{N}$ to be in the extreme Gumbel regime of Theorems~\ref{repeat}-\ref{repeat2} is
\begin{align}\label{theta2}
1
&\gg
\theta_{\text{gum}}:=
\begin{cases}
-\ln(\sigma^{(d-1)/2})/\ln N & \text{if }d\in\{2,3\},\,\kappa=\infty,\,\|X(0)\|=R,\\
-\ln(\kappa\sigma^{(d-1)/2})/\ln N & \text{if }d\in\{1,2,3\},\,\kappa<\infty,\,\|X(0)\|=R.
\end{cases}
\end{align}
Finally, if $X(0)=_{\dist}\text{Uniform}(\Omega)$, then we apply the condition in Remark~\ref{remarksuff1} to the Cases 3 and 4 above to find that a sufficient condition for $T_{N}$ to be in the extreme Weibull regime of Theorems~\ref{main}-\ref{moments} is
\begin{align}\label{theta3}
1
&\gg
\theta_{\text{wei}}:=
\begin{cases}
(dN\sigma^{d-1})^{-2} & \text{if }d\in\{1,2,3\},\,\kappa=\infty,\,X(0)=_{\dist}\text{Uniform},\\
(dN\kappa\sigma^{d-1})^{-1} & \text{if }d\in\{1,2,3\},\,\kappa<\infty,\,X(0)=_{\dist}\text{Uniform}.
\end{cases}
\end{align}

The conditions in \eqref{theta1}-\eqref{theta3} allow us to estimate the distribution of $T_{N}$ based on the values of $d$, $\kappa$, $\sigma$, and the initial searcher distribution. Indeed, if $\theta\in(0,1)$ is some small threshold parameter, then we can solve \eqref{theta1} for $N$ to find that $\theta_{\text{exp}}\le\theta$ if
\begin{align}\label{nexp}
N
\le N_{\text{exp}}(\theta)
:=
\begin{cases}
-\ln(\sigma)\big(2W_{0}(-\ln(\sigma)/(2\theta))\big)^{-1} & \text{if }d=2,\,\kappa=\infty,\\
\big(3\sigma W_{0}(1/(3\sigma\theta))\big)^{-1} & \text{if }d=3,\,\kappa=\infty,\\
\big(d\kappa\sigma^{d-1} W_{0}(1/(d\kappa\sigma^{d-1}\theta))\big)^{-1} & \text{if }d\in\{1,2,3\},\,\kappa<\infty,
\end{cases}
\end{align}
where $W_{0}(z)$ denotes the principal branch of the LambertW function \cite{corless1996} (defined as the inverse of $f(z) = ze^{z}$ and also called the product logarithm function).
In particular, if $N\le N_{\text{exp}}$, then $N$ is ``sufficiently small'' so that $T_{N}\approx_{\dist}{\textup{Exp}(\E[\tau]/N)}$.

Similarly, $\theta_{\text{gum}}\le\theta$ if
\begin{align}\label{ngum}
N\ge N_{\text{gum}}(\theta)
:=
\begin{cases}
\big(\sigma^{(d-1)/2}\big)^{-1/\theta} & \text{if }d\in\{1,2,3\},\,\kappa=\infty,\,\|X(0)\|=R,\\
\big(\kappa\sigma^{(d-1)/2}\big)^{-1/\theta} & \text{if }d\in\{1,2,3\},\,\kappa<\infty,\,\|X(0)\|=R.
\end{cases}
\end{align}
Finally, $\theta_{\text{wei}}\le\theta$ if
\begin{align}\label{nwei}
N\ge N_{\text{wei}}(\theta)
:=
\begin{cases}
\big(d\sigma^{d-1}\sqrt{\theta}\big)^{-1} & \text{if }d\in\{1,2,3\},\,\kappa=\infty,\,X(0)=_{\dist}\text{Uniform},\\
\big(d\kappa\sigma^{d-1}\theta\big)^{-1} & \text{if }d\in\{1,2,3\},\,\kappa<\infty,\,X(0)=_{\dist}\text{Uniform}.
\end{cases}
\end{align}
That is, \eqref{ngum} and \eqref{nwei} represent the $N$ ``sufficiently large'' values for which the extreme regimes of Section~\ref{fast} are valid. We note that the extreme Weibull regime for $\kappa<\infty$ is in fact exponential.

\begin{figure}[t]
\centering
\includegraphics[width=.495\linewidth]{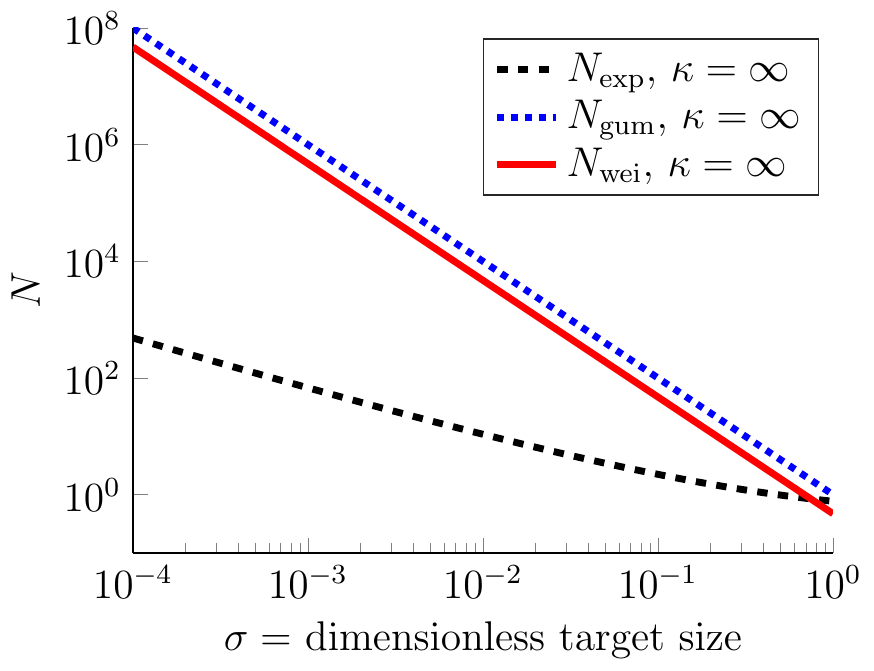}
\includegraphics[width=.495\linewidth]{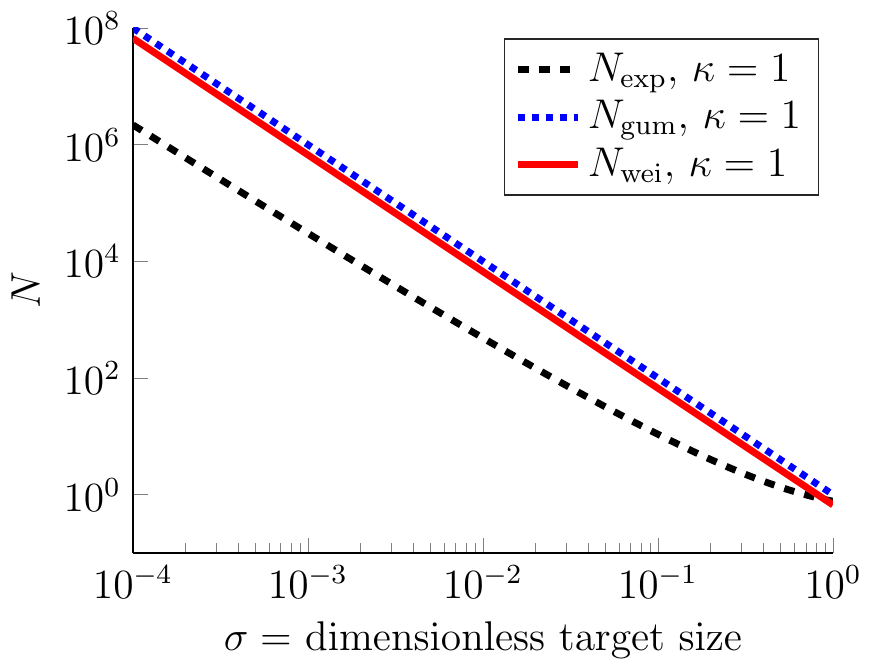}
\caption{
The curves in both panels are $N_{\text{exp}}(\theta)$, $N_{\text{gum}}(\theta)$, and $N_{\text{wei}}(\theta)$ in \eqref{nexp}-\eqref{nwei} as functions of $\sigma$ for $\theta=1/2$ and $d=3$. Hence, $T_{N}\approx_{\dist}\text{Exp}(\E[\tau]/N)$ in the region of $(\theta,N)$-parameter space below the black dashed curve. Similarly, $T_{N}$ is approximately Gumbel (respectively Weibull) in the region of $(\theta,N)$-parameter space above the blue dotted curve (respectively red solid curve). In the right panel, the Weibull distribution for large $N$ is actually an exponential distribution (see Section~\ref{section uniform imperfect}).
}
\label{figregime}
\end{figure}

In Figure~\ref{figregime}, we plot $N_{\text{exp}}$, $N_{\text{gum}}$, and $N_{\text{wei}}$ as functions of $\sigma\in(0,1)$ for $\theta=1/2$ and $d=3$ (the left panel is for $\kappa=\infty$ and the right panel is for $\kappa=1<\infty$). In these panels, the region of $(\sigma,N)$-parameter space below $N_{\text{exp}}$ is the regime in which $T_{N}$ is exponential, and the region above $N_{\text{gum}}$ (respectively $N_{\text{wei}}$) is the region in which $T_{N}$ is in the extreme Gumbel (respectively Weibull) regime if $\|X(0)\|=R$ (respectively $X(0)=_{\dist}\text{Uniform}$).

\subsection{Comparison to numerical simulations}\label{numerical}

\begin{figure}[t]
\centering
\includegraphics[width=.495\linewidth]{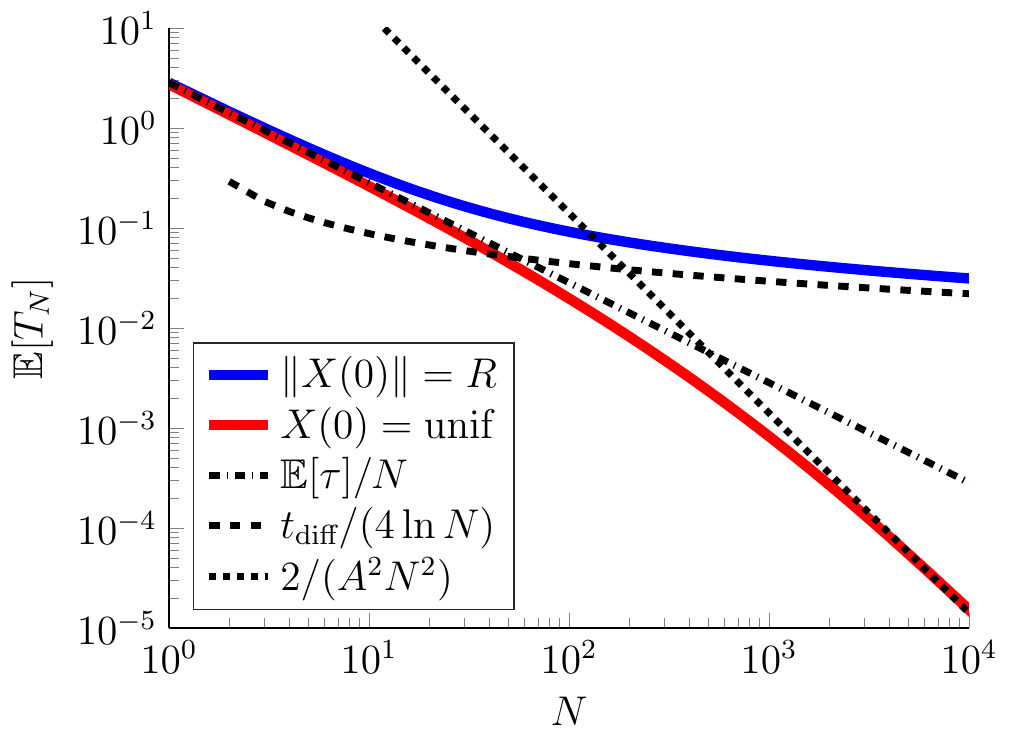}
\includegraphics[width=.495\linewidth]{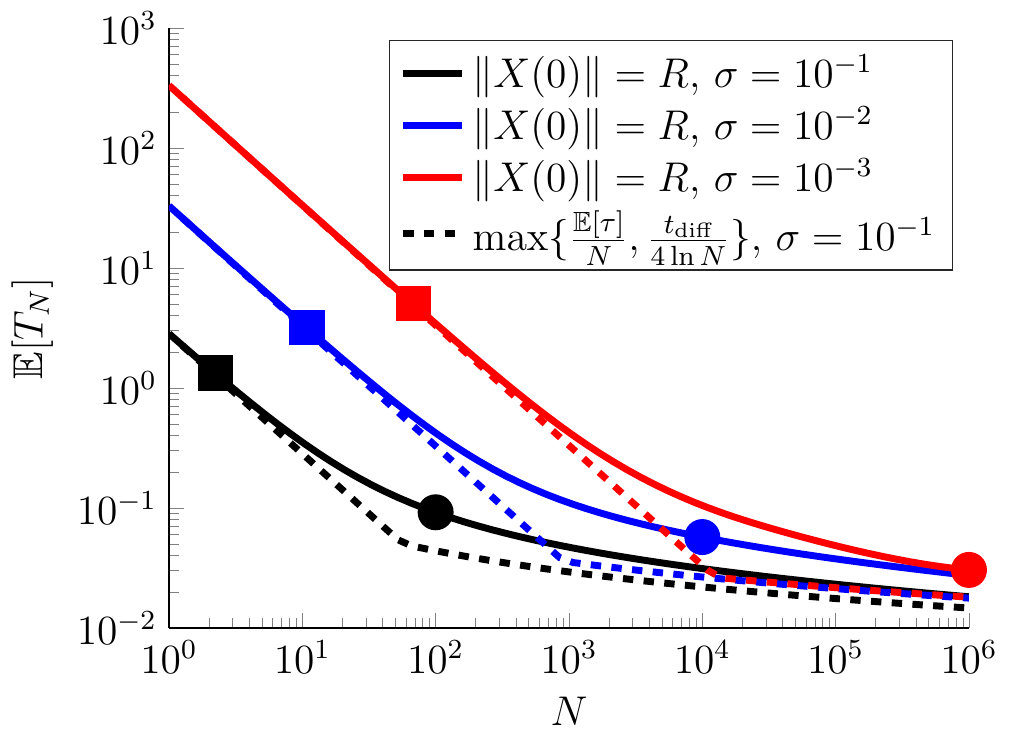}
\caption{Perfectly absorbing targets ($\kappa=\infty$). The left panel plots numerically computed values of $\E[T_{N}]$ for Cases 1 and 3 as solid blue and red curves, respectively. The black dashed, dotted, and dot-dashed curves are the theoretical asymptotic behaviors of Sections~\ref{section perfect} and \ref{section uniform perfect}. In this left panel, $\sigma=0.1$. In the right panel, the solid black, blue, and red curves are numerically computed values of $\E[T_{N}]$ for Case 1 and the dashed curves are the maximum of the corresponding theoretical behaviors for large and small $N$ (see \eqref{maxmax}). The squares are at $N=N_{\text{exp}}(\theta)$ (see \eqref{nexp}) which predict when $\E[T_{N}]$ transitions out of the exponential regime. Similarly, the circles are at $N=N_{\text{gum}}(\theta)$ (see \eqref{ngum}) which predict when $\E[T_{N}]$ transitions into the Gumbel regime. We take $\theta=1/2$ and $d=3$.}
\label{figzoo}
\end{figure}

In this section, we compare the analytical results of Sections~\ref{section perfect}-\ref{section compare 4} to numerical simulations. Numerically, we solve the partial differential equation \eqref{sann}-\eqref{bcimp} using the Matlab function \texttt{pdepe} \cite{matlab}. We then compute $\E[T_{N}]$ by numerical quadrature,
\begin{align*}
\E[T_{N}]=\int_{0}^{\infty}(S(t))^{N}\,\dd t,
\end{align*}
for both the Dirac delta function initial conditions of Sections~\ref{section perfect}-\ref{section imperfect} and the uniform initial conditions of Sections~\ref{section uniform perfect}-\ref{section uniform imperfect}.

Figure~\ref{figzoo} corresponds to the case of a perfectly absorbing target ($\kappa=\infty$) in Sections~\ref{section perfect} and \ref{section uniform perfect} (Cases 1 and 3). In the left panel of Figure~\ref{figzoo}, the solid curves are $\E[T_{N}]$ as a function of $N$ for  $\|X(0)\|=R$ (blue curve) and $X(0)=_{\dist}\text{Uniform}$ (red curve) and the black dashed, dotted, and dot-dashed curves are the theoretical asymptotic behaviors of Sections~\ref{section perfect} and \ref{section uniform perfect}. In agreement with the analysis, these numerical results show that (i) $\E[T_{N}]\approx\E[\tau]/N$ for small $N$ regardless of initial conditions, (ii) $\E[T_{N}]\sim\td/(4\ln N)$ as $N\to\infty$ if $\|X(0)\|=R$, and (iii) $\E[T_{N}]\sim2/(A^{2}N^{2})$ as $N\to\infty$ if $X(0)=_{\dist}\text{Uniform}$, where $A$ is in \eqref{apc333}.

The right panel of Figure~\ref{figzoo} plots $\E[T_{N}]$ as a function of $N$ for Case 1 ($\kappa=\infty$ and $\|X(0)\|=R$) for different values of $\sigma$ (the dimensionless target size). The squares are at $N=N_{\text{exp}}(\theta)$ (see \eqref{nexp}) and the circles are at $N=N_{\text{gum}}(\theta)$ (see \eqref{ngum}), both for $\theta=1/2$. In particular, these are the theoretical predictions for where $T_{N}$ transitions out of the exponential regime (squares) and where $T_{N}$ transitions into the Gumbel regime (circles), and these agree well with the numerical results. Interestingly, these figure shows that simply taking the maximum of the small $N$ and large $N$ behaviors is a good approximation for the mean fastest FPT,
\begin{align}\label{maxmax}
\E[T_{N}]
\approx\max\Big\{\frac{\E[\tau]}{N},\frac{\td}{4\ln N}\Big\}\quad\text{for all }N\ge1.
\end{align}

Figure~\ref{figkappa} corresponds to the case of a partially absorbing target ($\kappa<\infty$) in Sections \ref{section imperfect} and \ref{section uniform imperfect} (Cases 2 and 4). Here, the blue curves are for a small value of $\kappa$ (namely $\kappa=10^{-2}$) and the black curves are for a large value of $\kappa$ (namely $\kappa=10^{2}$). The red markers are the theoretical asymptotic behaviors of Sections~\ref{section imperfect} and \ref{section uniform imperfect}, which agree with the numerical results in the blue and black curves.

\begin{figure}[t]
\centering
\includegraphics[width=.75\linewidth]{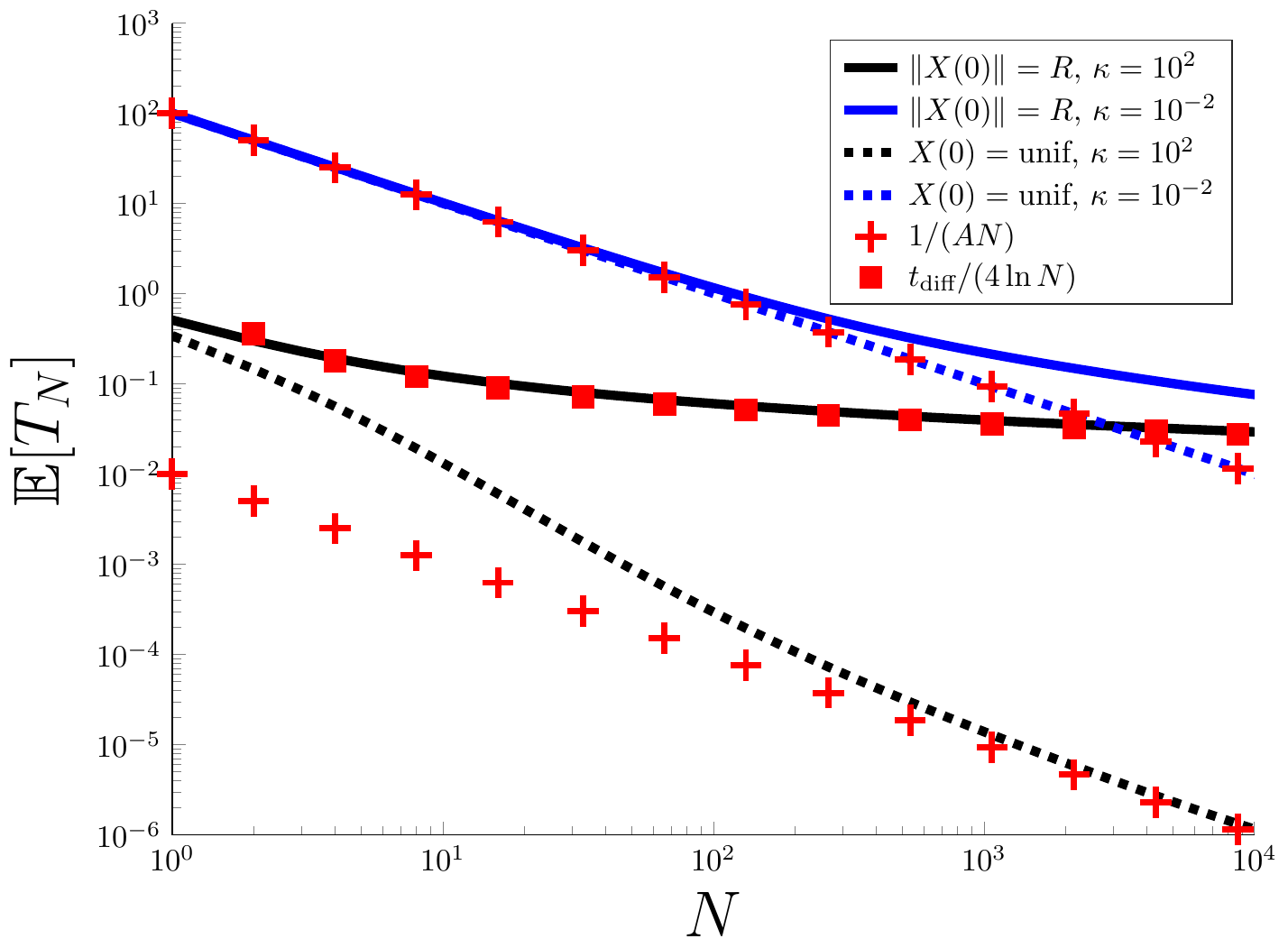}
\caption{Partially absorbing targets ($\kappa<\infty$). The blue curves are for a small value of $\kappa$ (namely $\kappa=10^{-2}$) and the black curves are for a large value of $\kappa$ (namely $\kappa=10^{2}$). The red markers are the theoretical asymptotic behaviors of Sections~\ref{section imperfect} and \ref{section uniform imperfect} (the two lines of red plusses are for the different values of $A$ corresponding to either $\kappa=10^{-2}$ or $\kappa=10^{2}$). We take $\sigma=0.1$ and $d=1$. 
}
\label{figkappa}
\end{figure}

\section{Discussion}

Much of the existing theory of FPTs of single diffusive searchers has been developed in the case that the FPT $\tau $ is much slower than the characteristic diffusion timescale $\td$. Mathematically, one typically introduces a small parameter $\eps>0$ (which, for example, measures the size of the target in the narrow escape problem \cite{holcman2014} or the strength of the noise for escape from a potential well \cite{hanggi1990}), and studies how the FPT $\tau=\tau(\eps)$ diverges as $\eps\to0$. In the case of $N$ searchers which reach the target at independent and identically distributed times $\{\tau_{1},\dots,\tau_{N}\}$, the fastest FPT,
\begin{align*}
T_{N}
=T_{N}(\eps):=\min\{\tau_{1}(\eps),\dots,\tau_{N}(\eps)\},
\end{align*}
also diverges as $\eps\to0$,
\begin{align}\label{slowd}
T_{N}(\eps)/\td\to\infty\quad\text{as }\eps\to0.
\end{align}
On the other hand, the fastest FPT vanishes in the many searcher limit,
\begin{align}\label{fastd}
T_{N}(\eps)/\td\to0\quad\text{as }N\to\infty.
\end{align}
Equations~\eqref{slowd} and \eqref{fastd} hold with probability one.

In this paper, we investigated the competition between the slow regime in \eqref{slowd} and the fast regime in \eqref{fastd}. We derived a simple sufficient condition (see \eqref{cond2}) and a simple necessary condition (see \eqref{condn}) for $T_{N}$ to be in the slow regime in \eqref{slowd}, based on the MFPT of a single searcher (the necessary condition also assumes that the initial searcher distribution satisfies \eqref{away}). These conditions quantify how $T_{N}(\eps)$ is in the slow regime in \eqref{slowd} for ``$N$ sufficiently small.'' If this sufficient condition is satisfied, then we gave an approximation for the full distribution and moments of $T_{N}$, and more generally of the $k$th fastest FPT, $T_{k,N}$.

We also gave sufficient conditions for the fast regime in \eqref{fastd} (see Remarks~\ref{remarksuff1} and \ref{remarksuff2}) and the limiting distribution and asymptotic moments of $T_{N}$ in the $N\to\infty$ limit. This analysis revealed the critical effect that initial conditions and target reactivity can have on the large $N$ distribution of $T_{N}$. Indeed, $T_{N}$ may be asymptotically Weibull, Gumbel, or exponential, and $\E[T_{N}]$ may decay like the reciprocal of $N^{2}$, $N$, and $\ln N$ as $N\to\infty$, depending on initial conditions and target reactivity. These various parameter regimes are summarized in Table~\ref{tablesummary}.

\begin{table}
\begin{center}
\footnotesize
\begin{tabular}{
| m{.2\textwidth}
| m{.28\textwidth}
| m{.37\textwidth} | } 
\hline
Parameter regime
&
$T_{N}$ mean and distribution & Comments\\ 
\hline\hline
\begin{align*}
N\exp\Big(\frac{-\E[\tau]}{N\td}\Big)\ll1
\end{align*}
&
\begin{align*}
\E[T_{N}]
\approx\E[\tau]/N,\\
T_{N}
\approx_{\dist}{\textup{Exp}(\E[\tau]/N)}
\end{align*}
&
The given parameter regime is a sufficient condition for the given exponential distribution of $T_{N}$.\\ 
\hline
\begin{align*}
4\ln N\frac{\E[\tau]}{N\td}\not\ll1
\end{align*}
&
\begin{align*}
\E[T_{N}]
\approx\E[\tau]/N,\\
T_{N}
\approx_{\dist}{\textup{Exp}(\E[\tau]/N)}
\end{align*}
&
The given parameter regime is a necessary condition for the given exponential distribution of $T_{N}$ if 
$U_{0}\cap \overline{U_{\text{T}}}=\varnothing$ (see \eqref{away}).
\\ 
\hline
\begin{align*}
\frac{1}{(AN)^{1/p}}\ll1
\end{align*}
&
\begin{align*}
\E[T_{N}]
\approx \frac{\Gamma(1+1/p)}{(AN)^{1/p}},\\
T_{N}\approx_{\dist}\textup{Weibull}\Big(\frac{1}{(AN)^{1/p}},p\Big)
\end{align*}
&
The given parameter regime is a sufficient condition for the given Weibull distribution of $T_{N}$ if
\begin{align*}
\P(\tau\le t)\sim At^{p}\quad\text{as }t\to0+,
\end{align*}
for some $A>0$, $p>0$, which is typical if $U_{0}\cap \overline{U_{\text{T}}}\neq\varnothing$. \\ 
\hline
\begin{align*}
\Big|\frac{\ln(AC^{p})+\gamma}{\ln N}\Big|\ll1
\end{align*}
&
\begin{align*}
\E[T_{N}]
\approx b_{N}-\gamma a_{N}\approx\frac{C}{\ln N},\\
T_{N}
\approx_{\dist}\textup{Gumbel}(b_{N},a_{N})
\end{align*}
&
The given parameter regime is a sufficient condition for the given Gumbel distribution of $T_{N}$ if
\begin{align*}
\P(\tau\le t)\sim At^{p}e^{-C/t}\quad\text{as }t\to0+,
\end{align*}
for some $C>0$, $A>0$, $p\in\R$, which is typical if $U_{0}\cap \overline{U_{\text{T}}}=\varnothing$. See Theorem~\ref{repeat} for $a_{N},b_{N}$.\\ 
\hline
\end{tabular}
\normalsize
\end{center}
\caption{
Summary of slow and fast parameter regimes. The first two rows correspond to the slow regime studied in section~\ref{slow} (small $N$). The last two rows correspond to the fast regime studied in section~\ref{fast} (large $N$). In the bottom row, $\gamma\approx0.5772$ is the Euler-Mascheroni constant.
}
\label{tablesummary}
\end{table}

Many authors have investigated extreme FPTs of diffusive searchers \cite{weiss1983, yuste1996, yuste2000, yuste2001, van2003, redner2014, meerson2015, basnayake2019, lawley2020esp, lawley2020uni, lawley2020dist} (see also \cite{godec2016x, grebenkov2018strong, hartich2018, hartich2019, hartich2019reaction} for interesting related work). Most of these prior works assume that the searchers all start at some fixed location, and they study the large $N$ asymptotics of $T_{N}$. For example, it is known that if each searcher starts at some fixed point $x$, then \cite{lawley2020uni} 
\begin{align}\label{old}
\E[T_{N}]
\sim\frac{L^{2}(x,\partial\Omega_{\text{T}})}{4D\ln N}\quad\text{as }N\to\infty,
\end{align}
where $L(x,\partial\Omega_{\text{T}})>0$ is a certain geodesic distance from $x$ to the target. A notable exception is in the first study of extreme FPTs of diffusion \cite{weiss1983}, in which Weiss, Shuler, and Lindenberg found (among many other things) that for diffusive searchers in one space dimension that start uniformly distributed with a perfectly absorbing target, $\E[T_{N}]$ decays like $1/N^{2}$ as $N\to\infty$ (see their equation~(3.15), which agrees with our equation~\eqref{mean3} with $d=1$ upon noting that their problem has targets at both ends of the interval). It is important to note that if the initial searcher distribution is given by some measure $\mu_{0}$, then the asymptotics of $\E[T_{N}]$ are not found merely by integrating \eqref{old} over $x$. That is, if $\mu_{0}$ is not a Dirac delta function, then 
\begin{align}\label{notintegrate}
\E[T_{N}]
\not\sim\int_{\Omega}\frac{L^{2}(x,\partial\Omega_{\text{T}})}{4D\ln N}\,\dd \mu_{0}(x)\quad\text{as }N\to\infty,
\end{align}
since, for example, $\E[T_{N}]$ can decay as $1/N^{2}$ or $1/N$ as we have shown. 

We close by discussing our results in the context of the recently formulated ``redundancy principle'' for biological systems \cite{schuss2019}. As described in the Introduction, the redundancy principle claims that the many seemingly redundant copies of an object (cells, proteins, molecules, etc.)\ are not a waste, but rather have the specific function of accelerating activation rates \cite{schuss2019}. That is, a biological system can overcome the prohibitively slow FPTs associated with diffusive search by deploying many searchers (i.e.\ it can move from the slow regime in \eqref{slowd} to the fast regime in \eqref{fastd}).

This principle was formulated in the context of the $1/\ln N$ decay of $\E[T_{N}]$ as $N\to\infty$, which is valid for searchers that cannot start arbitrarily close to the target (see \eqref{away}). However, in this case we have shown that $\E[T_{N}]$ initially decays as $1/N$ (specifically, $\E[T_{N}]\approx\E[\tau]/N$), and does not transition to the $1/\ln N$ regime (specifically, $\E[T_{N}]\approx\td/(4\ln N)$) until very large values of $N$ if $\eps\ll1$.

Therefore, adding additional searchers to a system that is in the $1/N$ regime accelerates the FPT $T_{N}$ to a much greater degree compared to adding additional searchers to a system that is already in the $1/\ln N$ regime. That is, the marginal benefit of additional searchers decreases sharply as the system goes from the $1/N$ regime to the $1/\ln N$ regime. Therefore, from a cost/benefit perspective (in which a system balances the cost of additional searchers with the benefit of faster activation \cite{rusakov2019, coombs2019}), our results predict that one should find more systems in the $1/N$ regime rather than deep into the $1/\ln N$ regime.

Finally, it is interesting to note the contrasting situation that occurs if the searchers are initially uniformly distributed (which is often assumed in studies of the narrow escape problem \cite{grebenkov2017}). In this case, $\E[T_{N}]$ transitions from $1/N$ decay to the faster $1/N^{2}$ decay as $N$ grows (for perfectly absorbing targets). Hence, the marginal benefit of additional searchers increases as a system moves from the $1/N$ regime to the $1/N^{2}$ regime.

\section{Appendix}

\subsection{Proofs}\label{proofs}

\begin{proof}[Proof of Proposition~\ref{propgen}]
For an $N$-dimensional vector $\x=(x_{1},\dots,x_{N})\in\R^{N}$, define the function
\begin{align*}
g(\x)
:=(x_{(1)},x_{(2)},\dots,x_{(n-1)},x_{(N)}),
\end{align*}
where $x_{(1)}\le x_{(2)}\le\dots\le x_{(N)}$. In words, $g$ sorts the elements in a vector from smallest to largest values. By the continuous mapping theorem (see, for example, Theorem~2.7 in \cite{billingsley2013}) and the definition of $T_{k,N}$, we have that 
\begin{align*}
&\lambda(T_{1,N},T_{2,N},\dots,T_{N-1,N},T_{N,N})\\
&\quad=g(\lambda(\tau_{1},\dots,\tau_{N}))
\to_{\dist}g(Y_{1},\dots,Y_{N})\quad\text{as }\eps\to0,
\end{align*}
where $\{Y_{n}\}_{n=1}^{N}$ are iid with $Y_{n}=_{\dist}\textup{Exp}(1)$. It is a classical result in order statistics \cite{renyi1953} that
\begin{align*}
g(Y_{1},\dots,Y_{N})
=_{\dist}\bigg(\frac{X_{1}}{N},
\frac{X_{1}}{N}+\frac{X_{2}}{N-1}
,\dots,
\sum_{j=1}^{N}\frac{X_{j}}{N-j+1}\bigg),
\end{align*}
where $\{X_{n}\}_{n=1}^{N}$ are iid with $X_{n}=_{\dist}\textup{Exp}(1)$, which completes the proof.
\end{proof}

\begin{proof}[Proof of Theorem~\ref{main}]
The proof of this theorem is similar to the proof of Theorem 2 in \cite{lawley2020pdmp}. Define $X_{n}=-\tau_{n}$ for $n\ge1$. Therefore, \eqref{short} implies that if $-1\ll x<0$, then
\begin{align}\label{F}
\begin{split}
F(x)
&:=\P(X_{n}\le x)
=\P(\tau_{n}\ge-x)
=1-{{A}}(-x)^{p}+o((-x)^{p})\quad\text{as }x\to0-.
\end{split}
\end{align}
Hence, if $y>0$, then \eqref{F} implies
\begin{align*}
\lim_{t\to0+}\frac{1-F(-ty)}{1-F(-t)}
=y^{p}.
\end{align*}
Therefore, Theorem 1.2.1 and Corollary 1.2.4 in \cite{haanbook} imply that
\begin{align*}
M_{N}:=\max\{X_{1},\dots,X_{N}\}
=-\min\{-X_{1},\dots,-X_{N}\}
=-T_{N}
\end{align*}
satisfies
\begin{align}\label{pcd}
\frac{-M_{N}}{a_{N}}
=\frac{T_{N}}{a_{N}}
\to_{\dist}\textup{Weibull}(1,p)\quad\text{as }N\to\infty,
\end{align}
where $a_{N}:=({{A}} N)^{-1/p}>0$.
\end{proof}

\begin{proof}[Proof of Theorem~\ref{moments}]
The result follows from Theorem~4 in \cite{lawley2020pdmp}.
\end{proof}

\begin{proof}[Proof of Theorem~\ref{kth}]
The result follows from Theorem~\ref{main} above and Theorem 3.5 in \cite{colesbook}.
\end{proof}

\begin{proof}[Proof of Theorem~\ref{kth moment}]
The result follows from Theorem~7 in \cite{lawley2020pdmp}.
\end{proof}

\subsection{Annular domains}

In this section, we determine the short-time behavior of the survival probabilities studied in Sections~\ref{section perfect}-\ref{section uniform imperfect}. The method is to solve for the Laplace transformed survival probability exactly and then determine the short-time behavior from the asymptotic behavior of the Laplace transform. This method has been employed in, for example, Reference~\cite{grebenkov2019full}.

\subsubsection{Perfect absorption, $\kappa=\infty$}

We first consider the case of a perfectly absorbing target. By taking the Laplace transform of \eqref{sr},
\begin{align*}
\SS(r,s)
:=\int_{0}^{\infty}e^{-st}\S(r,t)\,\dd t,
\end{align*}
and nondimensionalizing time $t\to\frac{D}{R^{2}}t$ and space $r\to\frac{1}{R}r$, we obtain that \eqref{sann}-\eqref{bcimp} becomes the dimensionless problem,
\begin{align}
s\SS-1
&=\big(\tfrac{d-1}{r}\tfrac{\dd}{\dd r}+\tfrac{\dd^{2}}{\dd r^{2}}\big)\SS,\quad s>0,\,r\in(\sigma,1),\label{lt}\\
\tfrac{\partial}{\partial r}\SS
&=0,\quad r=1,\label{bc1}\\
\SS
&=0,\quad r=\sigma,\label{bc2}
\end{align}
where $\sigma:=a/R\in(0,1)$ is the dimensionless target radius.

The general solution to \eqref{lt} is
\begin{align}\label{gensoln}
{\SS}(r,s)
=\frac{1}{s}+C_1I_{0,d}(\sqrt{s}r)+C_2K_{0,d}(\sqrt{s}r),
\end{align}
where $I_{\alpha,d}(x)$ and $K_{\alpha,d}(x)$ are the $d$-dimensional modified Bessel functions of order $\alpha$ when $d=2$, and
\begin{align*}
I_{\alpha,1}(x)
&:=e^{x},\quad
K_{\alpha,1}(x)
:=e^{-x}\\
I_{0,3}(x)
&:=\frac{\sinh(x)}{x},\quad
I_{1,3}(x)
:=\frac{x\cosh(x)-\sinh(x)}{x^{2}},\\
K_{0,3}(x)
&:=\frac{e^{-x}}{x},\quad
K_{1,3}(x)
:=\frac{e^{-x}(x+1)}{x^{2}}.
\end{align*}
Applying the boundary conditions in \eqref{bc1}-\eqref{bc2}, the solution \eqref{gensoln} becomes
\begin{align}\label{soln0}
{\SS}(r,s)
=\frac{1}{s}\Big[1-\frac{g(r,s)}{g(s,\eps)}\Big].
\end{align}
To find the behavior of \eqref{soln0} as $s\to\infty$, note that
\begin{align}\label{recall}
\begin{split}
I_{\alpha,d}(x)
&\sim\widetilde{I}_d\Big(\frac{1}{x}\Big)^{\frac{d-1}{2}}e^x\quad\text{as }x\to\infty, \\
K_{\alpha,d}(x)
&\sim\widetilde{K}_d\Big(\frac{1}{x}\Big)^{\frac{d-1}{2}}e^{-x}\quad\text{as }x\to\infty,
\end{split}
\end{align}
where $I_{d}'$ and $K_{d}'$ are constants determined by $d$. Hence, 
\eqref{soln0} has the large $s$ expansion,
\begin{align*}
    {\SS}(r,s)=\frac{1}{s}\Big[1-\Big(\frac{\eps}{r}\Big)^{\frac{d-1}{2}}e^{\sqrt{s}(\eps-r)}+o\big(e^{\sqrt{s}(\eps-r)}\big)\Big]
    \quad\text{as }s\to\infty.
\end{align*}
If the searcher is initially uniformly distributed in the domain $\Omega$, then we integrate over $r\in(\sigma,1)$ and obtain
\begin{align*}
\SS^{\unif}(r)
&=\int_{\sigma}^{1}\SS(r,s)\,\frac{r^{d-1}}{1-\sigma^{d}}\dd r\\
&=\frac{1}{s}\Big[1-\frac{d\sigma^{d-1}}{1-\sigma^{d}}\frac{1}{\sqrt{s}}
+o\Big(\frac{1}{\sqrt{s}}\Big)\Big]\quad\text{as }s\to\infty.
\end{align*}
Taking the inverse Laplace transform, we obtain the short time behavior,
\begin{align*}
1-S(r,t)
&\sim\frac{2}{\sqrt{\pi}(r-\sigma)}\Big(\frac{\sigma}{r}\Big)^{(d-1)/2}t^{1/2}e^{-(r-\sigma)^{2}/(4t)}\quad\text{as }t\to0+,\\
1-S^{\unif}(t)
&\sim\frac{2d\sigma^{d-1}}{\sqrt{\pi}(1-\sigma^{d})}t^{1/2}\quad\text{as }t\to0+.
\end{align*}

\subsubsection{Partially absorbing target, $\kappa<\infty$}

For the case of a partially absorbing target, we have that the Laplace transform,
\begin{align*}
\SS_{\imp}(r,s):=\int_{0}^{\infty}e^{-st}\S(r,t)\,\dd t,
\end{align*}
satisfies \eqref{lt}-\eqref{bc1}, and \eqref{bc2} is replaced by
\begin{align}\label{bc3}
\tfrac{\partial}{\partial r}\SS_{\imp}=\overline{\kappa} \SS_{\imp},\quad r=\sigma,
\end{align}
where $\overline{\kappa}:=\frac{\kon R}{D}$ is the dimensionless reactivity (which is a slightly differently nondimensionalization than \eqref{kappaconv}).

Applying the boundary conditions in \eqref{bc1} and \eqref{bc3} to the general solution in \eqref{gensoln} yields
\begin{align*}
\SS_{\imp}(r,s)
=\frac{1}{s}\Big[1+\frac{\overline{\kappa} g(r,s)}{{\tfrac{\partial g}{\partial r}}(s,\eps)-\overline{\kappa} g(s,\eps)}\Big],
\end{align*}
where
\begin{align*}
g(r,s)
&:=K_{1,d}(\sqrt{s})I_{0,d} (\sqrt{s}r )+I_{1,d} (\sqrt{s} )K_{0,d} (\sqrt{s}r ),\\
{\tfrac{\partial g}{\partial r}}(r,s)
&=\sqrt{s}\big[K_{1,d} (\sqrt{s} )I_{1,d} (\sqrt{s}r )-I_{1,d} (\sqrt{s} )K_{1,d} (\sqrt{s}r )\big].
\end{align*}
Using \eqref{recall}, we thus obtain the large $s$ expansion,
\begin{align*}
\SS_{\imp}(r,s)
=\frac{1}{s}\Big[1-\frac{\overline{\kappa}}{\sqrt{s}}\Big(\frac{\eps}{r}\Big)^{\frac{d-1}{2}}e^{\sqrt{s}(\eps-r)}+o\big(\frac{1}{\sqrt{s}}e^{\sqrt{s}(\eps-r)}\big)\Big]\quad\text{as }s\to\infty.
\end{align*}
If the searcher is initially uniformly distributed in the domain $\Omega$, then we integrate over $r\in(\sigma,1)$ and obtain
\begin{align*}
\SS_{\imp}^{\unif}(r)
&=\int_{\sigma}^{1}\SS_{\imp}(r,s)\,\frac{r^{d-1}}{1-\sigma^{d}}\dd r\\
&=\frac{1}{s}\Big[1-\frac{d\sigma^{d-1}}{1-\sigma^{d}}\frac{1}{\sqrt{s}}
+o\Big(\frac{1}{\sqrt{s}}\Big)\Big]\quad\text{as }s\to\infty.
\end{align*}
Taking the inverse Laplace transform, we obtain the short time behavior,
\begin{align*}
1-S_{\imp}(r,t)
&\sim\frac{4\overline{\kappa}}{\sqrt{\pi}(r-\sigma)^{2}}\Big(\frac{\sigma}{r}\Big)^{(d-1)/2}t^{3/2}e^{-(r-\sigma)^{2}/(4t)}\quad\text{as }t\to0+,\\
1-S_{\imp}^{\unif}(t)
&\sim\frac{d\overline{\kappa}\sigma^{d-1}}{1-\sigma^{d}}t\quad\text{as }t\to0+.
\end{align*}

\bibliography{library.bib}
\bibliographystyle{unsrt}

\end{document}